\documentclass[11pt,twoside,a4paper]{article}
\usepackage{amsmath,amssymb,mathrsfs}
\usepackage{amsthm}
\usepackage[utf8x]{inputenc}
\usepackage{bm}
\usepackage{avant}
\usepackage[english]{babel}
\usepackage{thmtools,thm-restate}
\usepackage[nottoc]{tocbibind}
\usepackage{hyperref}
\usepackage{graphicx}
\usepackage{enumitem}
\usepackage{tikz}
\usepackage{bbm}
\usepackage{esint}
\usepackage[inner=2.4cm,outer=2.4cm,top=2.8cm,bottom=2.8cm]{geometry}
\usepackage{mathtools}
\usepackage{makeidx}
\usepackage{bold-extra}

\newcommand{\scal}[2]{\ensuremath{\langle #1 , #2 \rangle}} 
\newcommand{\norm}[1]{\left\lVert#1\right\rVert}

\newcommand{\Leb}{\mathscr{L}}

\newcommand{\N}{\mathbb{N}}
\newcommand{\R}{\mathbb{R}}

\newcommand{\de}{\ensuremath{\, \mathrm d}} 

\newcommand\restr[2]{{
  \left.\kern-\nulldelimiterspace 
  #1 
  \right|_{#2} 
  }}

\newcommand{\CD}{\mathsf{CD}}
\newcommand{\RCD}{\mathsf{RCD}}

\newcommand{\Ch}{\mathsf{Ch}}
\newcommand{\X}{\mathsf{X}}
\newcommand{\Ks}{\mathsf{KS}}
\newcommand{\ks}{\mathsf{ks}}
\newcommand{\Lip}{\mathsf {Lip}}
\newcommand{\Lipglob}{\textbf {Lip}}
\newcommand{\Lipfunc}{\text {Lip}}

\newcommand{\di}{\mathsf d} 
\newcommand{\m}{\mathfrak m} 

\usepackage{titlesec}
\titleformat{\section}
  {\scshape\large\centering}
  {\thesection}{0.5em}{}
\titleformat{\subsection}
  {\scshape\centering}
  {\thesubsection}{0.4em}{}

\title{\textbf{A canonical infinitesimally Hilbertian structure on locally Minkowski spaces}}
\date{}
\author{Mattia Magnabosco\thanks{Institut f\"ur Angewandte Mathematik, Universit\"at Bonn. Email: magnabosco@iam.uni-bonn.de}, Chiara Rigoni\thanks{Fakult\"at f\"ur Mathematik, Universit\"at Wien. Email: chiara.rigoni@univie.ac.at}}

\newtheoremstyle{remark}
        {10pt}
        {10pt}
        {}
        {}
        {\itshape}
        {.}
        {.4em}
        {}

\newtheoremstyle{proof}
        {10pt}
        {10pt}
        {}
        {}
        {\itshape}
        {.}
        {.4em}
        {}
        
\newtheoremstyle{definition}
        {10pt}
        {10pt}
        {}
        {}
        {\bfseries}
        {.}
        {.4em}
        {}

\newtheoremstyle{theorem}
        {10pt}
        {10pt}
        {\slshape}
        {}
        {\bfseries}
        {.}
        {.4em}
        {}

\theoremstyle{theorem}
\newtheorem{theorem}{Theorem}[section]
\newtheorem{prop}[theorem]{Proposition}
\newtheorem{corollary}[theorem]{Corollary}
\newtheorem{lemma}[theorem]{Lemma}

\theoremstyle{definition}
\newtheorem{definition}[theorem]{Definition}

\theoremstyle{remark}
\newtheorem{remark}[theorem]{Remark}

\theoremstyle{proof}
\newtheorem*{pro}{Proof}
 {\popQED\end{pro}}

\makeindex
\begin{document}
\maketitle

\begin{abstract}

The aim of this paper is to show the existence of a canonical distance $\di'$ defined on a \emph{locally Minkowski} metric measure space $(\X,\di,\m)$ such that:
\begin{itemize}
    \item[i)] $\di'$ is equivalent to $\di$,
    \item[ii)] $(\X, \di', \m)$ is infinitesimally Hilbertian.
\end{itemize}
This new regularity assumption on $(\X, \di,\m)$ essentially forces the structure to be locally similar to a Minkowski space and defines a class of metric measure structures which includes all the Finsler manifolds, and it is actually strictly larger. The required distance $\di'$ will be the intrinsic distance $\di_\Ks$ associated to the so-called \emph{Korevaar-Schoen energy}, which is proven to be a quadratic form. In particular, we show that the Cheeger energy associated to the metric measure space $(\X, \di_\Ks, \m)$ is in fact the Korevaar-Schoen energy.

\end{abstract}

\section{Introduction}

A fundamental problem in the study of metric measure spaces  
is to generalize in this non-smooth setting some of the analytic results valid in the Euclidean setting. The starting point of this generalization consists in finding a suitable notion of pointwise norm of the gradient of a function, having at disposal just a distance and not a differential structure. The first answers in this direction come from the theories proposed by \textsc{Heinonen} and \textsc{Koskela} in \cite{MR1654771} and by \textsc{Shanmugalingam} in \cite{MR1809341}. In the setting of doubling metric measure spaces supporting a Poincaré inequality, \textsc{Cheeger} in \cite{MR1708448} proposed a construction based on a relaxation procedure: starting with a core of Lipschitz functions, it is possible to introduce the notion of \emph{minimal generalized upper gradient} which plays the role of the norm of the gradient in the non-smooth setting. The work of \textsc{Cheeger} was then refined by \textsc{Ambrosio}, \textsc{Gigli} and \textsc{Savaré} in  \cite{MR3152751}, where they introduced the notions of \emph{weak upper gradient} and \emph{relaxed upper gradient}. However, this last slightly different and more sophisticated approach coincide with the original one proposed by Cheeger when they are used to generalize the notion of Dirichlet energy, which in the Euclidean setting is defined by 
\begin{equation*}
    \mathscr{D}(f) := \int_{\R^n} |\nabla f(x)|^2 \de x \qquad f \in W^{1,2}(\R^n).
\end{equation*}
In particular, the integral of the minimal generalized upper gradient (or relaxed/weak upper gradient) squared of a function $f$ defines the so-called Cheeger energy $\Ch$ of $f$, and provides a suitable generalization of the Dirichlet energy in the setting of metric measure spaces. Moreover, since the classical heat flow in Euclidean spaces can be seen as the gradient flow (cfr. \cite{MR2401600}) of the Dirichlet energy, also the gradient flow of the Cheeger energy defines a suitable notion of heat flow in the framework of metric measure spaces.\\

It turns out that the study of the properties of the heat flow actually plays a prominent role in order to build a differential structure on a metric measure space. In fact, the gradient of a Sobolev function is in general not uniquely defined and even if so it might not linearly depend on the function, as it happens on smooth Finsler manifolds. Spaces where the gradient of a Sobolev function $f$ is unique and linearly depends on $f$ are those which, from the Sobolev calculus point of view, resemble Riemannian manifolds among the more general Finsler ones. These can be characterized as those for which the heat flow is linear or, equivalently, the energy is a quadratic form. This motivates the following definition introduced by \textsc{Gigli} in \cite{MR3381131} in the setting of metric measure spaces:

\begin{definition}\label{def:infHil}
A metric measure space $(\X,\di,\m)$ is said to be \emph{infinitesimally Hilbertian} if the associated Cheeger energy is a quadratic form on $L^2(\X,\m)$, that is
\begin{equation*}
    \Ch[f+g]+ \Ch[f-g] = 2\Ch[f] + 2 \Ch[g] \qquad \forall f,g\in L^2(\X,\m).
\end{equation*}
\end{definition}
On such spaces, the tangent/cotangent module is, when seen as Banach space, an Hilbert space and its pointwise norm satisfies a pointwise parallelogram identity (see \cite{GigliND}).\\

This notion provides new tools in the investigation of the analytic and geometric properties of $\CD$ spaces, introduced by \textsc{Lott}- \textsc{Villani} \cite{MR2480619} and \textsc{Sturm} \cite{MR2237206,MR2237207}. An infinitesimally Hilbertian metric measure space satisfying the $\CD(K,N)$ condition, for some $K \in \R$ and $N\ge 1$, is said to be an $\RCD(K,N)$ space. The study of $\RCD$ spaces, that was pioneered by \textsc{Ambrosio}, \textsc{Gigli} and \textsc{Savaré} \cite{MR3152751,MR3205729,MR3298475}, has had many remarkable developments in the recent years. In particular, the class of $\RCD$ spaces includes all the $\CD$ spaces whose analytic structure resembles more the one of a Riemannian manifold: in this view, as shown by \textsc{Ohta} and \textsc{Sturm} in \cite{MR2917125}, $(\R^n, \norm{\cdot}, \Leb^n)$ is an $\RCD(0, n)$ space if and only if $\norm{\cdot }$ is an Euclidean norm.\\

Our framework consists of the class of metric measure spaces $(\X,\di,\m)$ which are doubling, support the Poincaré inequality and are \emph{locally Minkowski}, in the sense of Definition \ref{def:locMink}. Intuitively, this new assumption on the space requires the space to be locally ``almost" isometric to a Minkowski space (cfr. \cite{MR2917125}). The class of locally Minkowski spaces contains all the Finsler manifolds, but also more ``irregular" spaces. In fact, we do not require the existence of any smooth structure and we also allow a locally Minkowski space to have different topological dimensions in different regions, as shown in the example represented in Figure \ref{fig:example}.\\

The aim of this paper is to investigate the existence of a canonical distance $\di'$ on a given locally Minkowski metric measure space $(\X,\di,\m)$, satisfying the following properties: 
\begin{itemize}
    \item $\di'$ equivalent to $\di$, meaning that $c_1 \di \leq \di' \leq c_2 \di$, for some positive constants $c_1<c_2$,
    \item the metric measure space $(\X, \di', \m)$ is infinitesimally Hilbertian.
\end{itemize}
We are going to identify the desired distance $\di'$ as the intrinsic distance $\di_\Ks$ associated to the so-called \emph{Korevaar-Schoen energy}. This functional was first introduced in \cite{KS}, as the key tool to study the harmonicity of  maps defined from a smooth manifold to a general metric space (see also \cite{Jost}). We point out that recently this energy has been a central object of investigation in a series of works \cite{GT1, GN21, MR4316816}, where it was extended to maps defined on more general $\RCD$ spaces. Notice that since this new distance $\di'$ is defined intrinsically, it will be in particular canonical, meaning that it will depend only on the space and not on particular choices of some other geometric objects. As can intuitively be guessed from its definition (see Subsection \ref{sec:ksenergy}), the Korevaar-Schoen energy will turn out to be a quadratic form. 

It is important to underline that if we apply our procedure to a $\CD(K, N)$ space $(\X, \di, \m)$ which is locally Minkowski, the resulting space $(\X, \di', \m)$ will not necessarily satisfy any synthetic curvature-dimension bound (nor in particular it will be an $\RCD$ space, as explained in Section \ref{sec:preliminaries}).\\

In the last section we will show that the Cheeger energy associated to the distance $\di_\Ks$ is exactly the Korevaar-Schoen energy, proving in particular that $(\X, \di_\Ks, \m)$ is an infinitesimally Hilbertian space. Moreover, it was recently proven by \textsc{Gigli} and \textsc{Tyulenev} in \cite{MR4316816} that for $\RCD$ spaces the Korevaar-Schoen energy is equal (up to a dimensional constant) to the Cheeger energy: this justifies the choice of the Korevaar-Schoen energy as the quadratic form inducing the infinitesimally Hilbertian structure of the space $(\X, \di_{\Ks}, \m)$. We also point out that a metric associated to the Korevaar-Schoen energy was already studied in the setting of Finsler manifolds by \textsc{Centore}  in \cite{MR1829124}.\\

In order to prove that the distance $\di_\Ks$ is equivalent to $\di$, we will pass through the intrinsic distance $\di_\Ch$ associated to the Cheeger energy. It was already proven by \textsc{Ohta} in \cite{MR2231926} that, under some additional assumptions, the distance $\di_\Ch$ is equivalent to $\di$: in Section \ref{sec:diCh}, we will adapt his argument to our setting to obtain the same equivalence. On the other hand, the locally Minkowski assumption is essential to prove that $\di_\Ks$ is equivalent to $\di_\Ch$: we will show that this hypothesis  guarantees the convergence of the Korevaar-Schoen potentials, in a sense that will be clarified later (see Section \ref{sec:convpot}). This, combined with some refined estimated of the Korevaar-Schoen potentials, which in particular are improved from the ones in \cite{MR2237207}, will provide the equivalence of the distances $\di_\Ks\simeq\di_\Ch$.\\

Finally, we point out that a fundamental tool to prove that the Korevaar-Schoen energy is actually the Cheeger energy associated to $\di_\Ks$ is a version of the Rademacher theorem for locally Minkowski spaces (Proposition \ref{prop:differential}). This result is a refinement of the one proposed by \textsc{Cheeger} in \cite{MR1708448} for metric measure spaces satisfying some additional structural assumptions. The techniques used in our proof are a suitable adaptation of the ones developed by \textsc{Cheeger}, but applied to the setting of locally Minkowski spaces. As discussed in Remark \ref{rmk:nonunique} and contrary to its classical Euclidean version, the Rademacher theorem we provide is not a uniqueness result. In fact, we just prove the existence of a differential, which takes the form of a linear function on the tangent space for $\m$-almost every point. However, this will turn out to be sufficient for our purpose, in view of the convergence of the Korevaar-Schoen potentials.\\

\noindent {\bf Acknowledgments:} The authors thank Professor Karl-Theodor Sturm for many valuable discussions. The second author gratefully acknowledges support by the European Union through the ERC--AdG 694405 ``Ricci\-Bounds'' and by the Austrian Science Fund (FWF) through project  F65.

\section{Preliminaries}\label{sec:preliminaries}

In this first section we collect all the preliminary definitions and results we will need in the paper, outlining the setting in which we will work.

\begin{definition}\label{def:mms}
A triple $(\X,\di,\m)$ is said to be a metric measure space if $(\X,\di)$ is a complete and separable metric space and $\m$ is a Borel measure such that $\m(\X)<\infty$.
\end{definition}

\noindent The finiteness of the reference measure $\m$ is not always required, however we decided to do it this work. This assumption makes some proofs and discussions much easier and it does not really affect the generality of the results. Therefore we invite the reader to keep in mind that this is a technical requirement, rather than a fundamental one.

\begin{definition}[Doubling Condition]
A metric measure space $(\X,\di,\m)$ is said to satisfy the \textit{(local) doubling condition}, if for every constant $R>0$ there exists $C_D:=C_D(R)>0$ such that for every $x\in \X$ it holds
\begin{equation}\label{eq:doubling}
    \m(B_{2r}(x)) \leq C_D\m (B_{r}(x)) \quad \text{for every } 0<r< R.
\end{equation}
\end{definition}


\noindent A function $f:\X\to\R$ is Lipschitz ($f\in \Lipfunc(\X)$) if 
\begin{equation*}
    \Lipglob[f]:= \sup_{x,y\in \X} \frac{|f(x)-f(y)|}{\di(x,y)}<\infty,
\end{equation*}
this quantity will be called global Lipschitz constant. 
For a Lipschitz function $f$, we define the pointwise Lipschitz constant:
\begin{equation}\label{eq:pointlip}
    \Lip[f](\bar z) :=\limsup _{r \rightarrow 0} \sup _{z\in B_r( \bar z)} \frac{|f(z)-f(\bar z)|}{r} = \limsup_{z\to \bar z} \frac{|f(z)-f(\bar z)|}{\di(z,\bar z)}.
\end{equation}
Both the global Lipschitz constant and the pointwise Lipschitz constant are obviously strictly dependent on the reference distance $\di$, however we decided not to make this dependence explicit, in order to ease the notation. In the following we are going to consider different distances on the space $\X$, but, unless otherwise indicated, the notations $\Lipfunc(\X)$, $\Lipglob[f]$ and $\Lip[f]$ will refer to the objects defined with respect to the reference distance $\di$.
Given $\varepsilon>0$, a map $f:X \to Y$ between two metric spaces $(X,\di_X)$ and $(Y,\di_Y)$ is said to be $\varepsilon$-bi-Lipschitz if 
\begin{equation*}
    (1-\varepsilon)\cdot \di_X(x,x')\leq \di_Y(f(x),f(x')) \leq (1+\varepsilon)\cdot \di_Y(x,x')
\end{equation*}
for every $x,x'\in X$.

In a metric space $(\X,\di)$, we say that a curve $\gamma:[0,1] \to \X$ is absolutely continuous (and we write $\gamma\in AC([0,1], \X)$) if there exists $g\in L^1([0,1])$ such that 
\begin{equation}\label{eq:acdef}
    \di(\gamma_{s}, \gamma_{t}) \leq \int_{s}^{t} g(r) \de r \quad \forall s< t \in [0,1].
\end{equation}
It can be proven that, if $\gamma\in AC([0,1], \X)$ there exists a minimal function $g$ satisfying \eqref{eq:acdef} which is called \textit{metric derivative} and given for a.e. $t\in[0,1]$ by
\begin{equation*}
    |\dot{\gamma}_{t}|:=\lim _{h\to 0} \frac{\di(\gamma_{t+h}, \gamma_{t})}{|s-t|}.
\end{equation*}
The metric derivative plays the role of the velocity of the curve $\gamma$; in particular it holds $l(\gamma)= \int_0^1 |\dot{\gamma}_{t}|\de t$, where $l(\gamma)$ denotes the length of $\gamma$. A metric space is said to be a \emph{length space} if the distance between any two points is equal to the infimum of the lengths of all the absolutely continuous curves which join them. 

We recall the fundamental notion of \textit{upper gradient} that was introduced in \cite{MR1654771}.

\begin{definition}[Upper Gradient]
In a metric measure space $(\X,\di,\m)$, let $f: \X \to \Bar{\R}$ be a Borel measurable function. We say that $g:\X\to [0,\infty]$ is an upper gradient for $f$ if, for every $\gamma \in AC ([0,1], \X)$, the function $s \mapsto g(\gamma_s)|\Dot{\gamma_s}|$ is measurable and 
\begin{equation*}
    |f(\gamma_1)- f(\gamma_0)| \leq \int_0^1g(\gamma_s)|\Dot{\gamma_s}| \de s.
\end{equation*}
\end{definition}

Before going on we introduce the following (standard) notation for the mean integral, that will be used for the all paper:
\begin{equation*}
    f_B := \fint_B f \de \m := \frac{1}{\m(B)}\int_B f \de \m 
\end{equation*}
\begin{definition}[Poincaré Inequality]
We say that a metric measure space $(\X,\di,\m)$ supports a \textit{(local) Poincaré inequality} if there exists $1\leq \Lambda <\infty$ such that for every $R>0$, we can find a constant $C_P:=C_P(R)\geq 1$ for which the inequality
\begin{equation}\label{eq:poincare}
    \fint_{B_r(x)} |f-f_{B_r(x)}|\leq C_P r \left( \fint_{B_{\Lambda r}(x)} g^2 \de \m \right)^\frac 12
\end{equation}
holds for every measurable function $f$, every upper gradient $g$ of $f$ and every $0<r< R_P$. 
\end{definition}

Under this assumption we can provide another characterization of the pointwise Lipschitz constant due to Cheeger \cite[Corollary 6.36]{MR1708448} that will help us in the following. 
\begin{prop}\label{prop:pointlip}
Let $(\X,\di,\m)$ be a metric measure space satisfying the doubling condition and supporting a Poincar\'e inequality. Given $f\in\Lipfunc(\X)$, then for $\m$-almost every $x\in \X$ it holds 
\begin{equation*}
    \Lip[f](x) = \lim_{r\to 0} \sup_{\di(y,x)=r} \frac{|f(y)-f(x)|}{r}.
\end{equation*}
\end{prop}

\begin{prop}\label{prop:D+P}
 Let $(\X,\di,\m)$ be a metric measure space that satisfies a doubling condition and supports a Poincaré inequality. If $\di'$ is a distance on $\X$, equivalent to $\di$, the metric measure space $(\X,\di',\m)$ satisfies a doubling condition and supports a Poincaré inequality.
\end{prop}

\begin{proof}
Let $k\in \N$ be such that 
\begin{equation*}
   2^{-k} \di \leq \di ' \leq 2^k \di.
\end{equation*}
As a consequence, for every $x\in \X$ and $r>0$ we have that
\begin{equation*}
    B_{2^{-k}r}(x)\subseteq B'_{r}(x) \subseteq B'_{2r}(x) \subseteq B_{2^{k+1}r}(x).
\end{equation*}
Then, we can conclude that 
\begin{equation*}
  \m\big(B'_{r}(x)\big) \geq \m \big(B_{2^{-k}r}(x)\big) \geq \frac 1{C_D^{2k+1}}  \m \big( B_{2^{k+1}r}(x)\big) \geq \frac 1{C_D^{2k+1}} \m \big( B'_{2r}(x)\big),
\end{equation*}
where $C_D:=C_D(2^{k+1} r)$. This proves the doubling condition.\\
In order to prove the Poncar\'e, we fix $R'>0$, we consider a measurable function $f$ and a point $x\in \X$. Proceeding as done in the first part of the proof, we deduce that for every $R<R'$    
\begin{equation*}
    B'_{R}(x) \subseteq B_{2^{k}R}(x) \quad \text{and}\quad \m\big(B'_{R}(x)\big) \geq \frac 1{C_D^{2k}}  \m \big( B_{2^{k}R}(x)\big),
\end{equation*}
where $C_D:=C_D(2^{k} R')$.
As a consequence, using Lemma \ref{lem:meanint}, we can conclude that for every $R<R'$ it holds
\begin{align*}
    \fint_{B'_R(x)} |f - f_{B'_R(x)}| \de \m &= \frac1{\m(B'_R(x))} \int_{B'_R(x)} |f - f_{B'_R(x)}| \de \m  \\
    &\leq \frac2{\m(B'_R(x))} \int_{B'_R(x)} |f - f_{B_{2^{k} R}(x)}| \de \m \\&\leq \frac{2 C_D^{2k}}{ \m(B_{2^k R}(x))} \int_{B_{2^{k} R}(x)} |f - f_{B_{2^{k} R}(x)}| \de \m \\
    &= 2C_D^{2k} \fint_{B_{2^{k} R}(x)} |f - f_{B_{2^{k} R}(x)}| \de \m.
\end{align*}
On the other hand, it is easy to notice that for every upper gradient $g$ of $f$, with respect to $\di'$, $2^k g$ is an upper gradient of $f$, with respect to $\di$. Then we can apply the Poincar\'e inequality for the metric measure space $(\X,\di,\m)$, and deduce that
\begin{equation*}
    \fint_{B_{2^k R}(x)} |f - f_{B_{2^k R}(x)}| \leq C_P 2^k R \left( \fint_{B_{ 2^k\Lambda R}(x)} 2^{2k}g^2 \de \m \right)^\frac 12 = C_P 2^{2k} R \left( \fint_{B_{ 2^k\Lambda R}(x)}g^2 \de \m \right)^\frac 12,
\end{equation*}
where $C_P=C_P(2^k R')$. Moreover, proceeding as before, we can prove that 
\begin{equation*}
    \fint_{B_{ 2^k\Lambda R}(x)} g^2 \de \m \leq \big(C'_D(2^{2k}\Lambda R')\big)^{2k}\fint_{B'_{ 2^{2k}\Lambda R}(x)} g^2 \de \m,
\end{equation*}
where $C'_D(\cdot)$ denotes the (local) doubling constant with respect to the distance $\di'$. Putting together the last three estimates (which do not depend on $R<R'$, $f$ and $x$) we prove the desired inequality for every $\di$-upper gradient $g$.
\end{proof}

\begin{lemma}\label{lem:meanint}
 In a measure space $(\X,\m)$, let $f:\X\to \R$ be a measurable function. Then, for every measurable set $A$ and every constant $c \in \R$ the following inequality holds
 \begin{equation*}
      \int_A |f-f_{A}| \de \m  \leq 2\int_A |f-c| \de \m.
 \end{equation*}
\end{lemma}

\begin{proof}
It is easy to notice that we can assume $f_A=0$, without losing generality, then
\begin{equation*}
    \int_A f^+ \de \m = \int_A f^- \de \m = \frac12\int_A |f| \de \m ,
\end{equation*}
where $f^+= \max\{f,0\}$ and $f^-=\max\{-f,0\}$. Now, if $c\leq0$, we can define $A^+:=\{f=f^+\}$ and observe that
\begin{equation*}
    \int_A |f-c| \de \m \geq  \int_{A^+} |f-c| \de \m \geq  \int_{A^+} |f| \de \m = \int_A f^+ \de \m =\frac12\int_A |f| \de \m.
\end{equation*}
The analogous procedure for $c\geq0$ concludes the proof.
\end{proof}

We recall that the Lebesgue diffentiation theorem holds is metric measure spaces satisfying the doubling condition, many proofs in this work rely on this fundamental result. We refer the reader to \cite{MR1800917} for a proof of this result.

\begin{theorem}[Lebesgue Differentiation Theorem]
Let $(\X,\di,\m)$ be a metric measure space satisfying the doubling condition, then given $f\in L^1(\X,\m)$, for $\m$-almost every $x\in \X$ it holds that
\begin{equation*}
    \lim_{r\to 0} \fint_{B_r(x)} f(y) \de \m (y) = f(x).
\end{equation*}
In particular, for every measurable set $A\subseteq \X$, taking $f= \mathbf{1}_A$, we can deduce that for $\m$-almost every $x\in A$ 
\begin{equation*}
    \lim_{r\to 0 } \frac{\m(B_r(x)\cap A)}{\m (B_r(x))} = 1,
\end{equation*}
if this is satisfied we will call $x$ a density point of $A$.
\end{theorem}

Before presenting the notion of locally Minkowski space, we introduce the following notation, that will help us in the formulation: 
\begin{equation*}
    \di_r := \frac 1 r \cdot \di.
\end{equation*}
Moreover, we clarify that we say that a norm $\norm{\cdot}$ on $\R^n$ is $C^1$ if it is $C^1$ outside the origin.

\begin{definition}\label{def:locMink}
We say that $(\X,\di,\m)$ is a \emph{locally Minkowski space} if for $\m$-almost every $x\in\X$ there exist $n(x)\leq N <\infty$  and a family of maps 
\begin{equation*}
    \Big\{i^{x,r}: (B_r(x), \di_r) \to  (\R^{n(x)},\norm{\cdot})\Big\}_{r > 0}
\end{equation*}
where $\norm{\cdot}$ is a $C^1$-norm, satisfying the following properties:
\begin{enumerate}
    \item $i^{x,r}(x)=0$ for every $r>0$,
    \item for every $\varepsilon>0$ there exists $r(\varepsilon)$ such that for every $r < r(\varepsilon)$ the map $i^{x,r}$ is $\varepsilon$-bi-Lipschitz and $B_{1-\varepsilon}(0)\subseteq i^{x,r} (B_r(x))$; 
    \item there exists a constant $c(x)$ such that 
    \begin{equation}\label{eq:pfmea}
        (1-\varepsilon)c(x) \cdot \Leb^{n(x)}\leq (i^{x,r})_\# \m^{x,r} \leq (1+\varepsilon)c(x) \cdot \Leb^{n(x)},
    \end{equation}
    on the set $i^{x,r} (B_r(x))$, where $\m^{x,r} :=\frac{\m}{\m(B_r(x))}$.
\end{enumerate}
Every point where this property is satisfied will be called $\textit{regular}$ point and the set of regular points will be denoted by $\mathcal{R}(\X)$.
\end{definition}
\begin{remark}\label{rmk:deflocMink}
(i) Assumption 1 is not really significant, since it can always be achieved by translation, however, we ask it in order to simplify the last part of assumption 2. \\
(ii) It is easy to realize that $c(x)$ has an explicit representation as $c(x)= \Leb^{n(x)}(B_1(0))^{-1}$. This shows in particular that this constant depends only on $(\R^{n(x)},\norm{\cdot})$.\\
(iii) The requirement $B_{1-\varepsilon}(0)\subseteq i^{x,r} (B_r(x))$ is necessary for the locally Minkowski assumption to really prescribe (locally) the geometry of the metric measure space. In fact, without assuming it, the space would just be locally almost isometric to a subset of some normed space, while the hypothesis $B_{1-\varepsilon}(0)\subseteq i^{x,r} (B_r(x))$ basically guarantees the ``sujectivity of the charts $i^{x,r}$" and consequently the fact that the geometry of the space is locally similar to the one of $(\R^{n(x)},\norm{\cdot})$.\\
{(iv) It is clear from the definition that every locally Minkowski space is locally compact.}
\end{remark}

\begin{figure}[b]
\begin{center}

\begin{tikzpicture} [scale=0.7]

\fill[color=black!17!white] (0,0)--(5.5,2.75)--(5.5,-2.75)--cycle;
\draw[thick](-6,0)--(0,0);
\draw[thick](0,0)--(5,2.5);
\draw[thick](0,0)--(5,-2.5);
\draw[thick,dotted](5,2.5)--(6,3);
\draw[thick,dotted](5,-2.5)--(6,-3);
\node at (-3,0)[label=north:$\mathcal H^1$] {};

\node at (3,0)[label=east:$\Leb^2$] {};
\end{tikzpicture}

\caption{Example of a locally Minkowski space having non-constant dimension.}
\label{fig:example}
\end{center}
\end{figure}
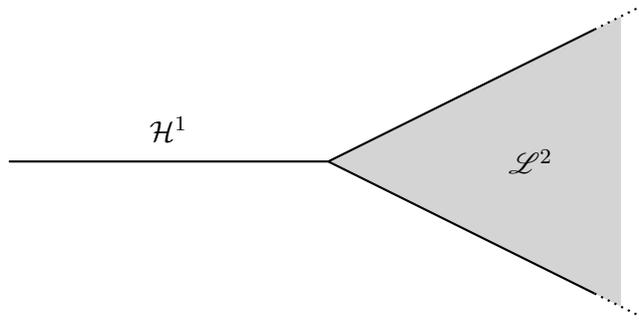

We bring the attention of the reader to the fact that a locally Minkowski space does not need to have constant (topological) dimension. In fact, consider the space $\X$ represented in Figure \ref{fig:example} obtained by gluing a ray and a closed cone in $\R^2$. We endow this space with the measure $\m$, obtained by summing the Hausdorff measure $\mathcal{H}^1$ on the ray and the Lebesgue measure $\Leb^2$ on the cone, and with a distance $\di$ induced by a $C^1$-norm $\norm{\cdot}$ in the following way
\begin{equation*}
    \di(x,y) = \inf\left\{ \int_0^1 \norm{\dot{\gamma(t)}} \de t : \gamma\in AC([0,1],\X), \, \gamma(0)=x, \, \gamma(1)=y\right\}.
\end{equation*}
It is easy to realize that the resulting metric measure space $(\X,\di,\m)$ is locally Minkowski, indeed the set of regular points consists of all the points in the interior of the cone and all the points of the ray except for the vertex of the cone, which clearly has full $\m$-measure.

This observation is particularly interesting in relation to the work of {Bruè} and {Semola} in \cite{MR4156601}, according to which every $\RCD(K,N)$ space has constant dimension. This means that, if we start with a metric measure space $(\X,\di,\m)$ having non-constant dimension and we construct a distance $\di'$ equivalent to $\di$ and for which $(\X, \di', \m)$ is infinitesimally Hilbertian, the metric measure space $(\X,\di',\m)$ will not be $\RCD(K,N)$, even if $(\X,\di,\m)$ is a $\CD$ space. Notice that, according to \cite{magnabosco2021example}, it is possible to construct a $\CD$ space having non-constant dimension.

\subsection{The Cheeger Energy}

In the following, for a fixed $f \in L^2(\X, \m)$, we denote by $\{g_i\}_{i \in \N} \in L^2(\X, \m)$ a sequence of functions such that:
\begin{enumerate}
    \item  for every $i \in \N$, $g_i$ is an upper gradient of $f_i \in L^2(\X, \m)$, 
    \item the sequence $\{f_i\}_{i \in \N}$ is such that $f_i \to f$ in $L^2(\X, \m)$ as $i \to +\infty$.
\end{enumerate}
Hence for any $f \in L^2(\X, \m)$ we set
\[
|| f ||_{1, 2} := \| f \|_{L^2(\X, \m)} + \inf_{\{g_i\}} \liminf_{i \to \infty} || g_i ||_{L^2 (\X, \m)}.
\]

\begin{definition}[The Sobolev space $H_{1,2}(\X, \di, \m)$]
We define the \emph{Sobolev space} $H_{1,2}(\X, \di, \m)$ as the subspace
\[
H_{1, 2}(\X, \di, \m) := \{ f \in L^2(\X, \m) : || f ||_{1, 2} < \infty \}
\]
equipped with the norm $|| \cdot ||_{1, 2}$.
\end{definition}

\noindent Then $(H^{1, 2}(\X, \di, \m), \| \cdot \|_{1, 2})$ is a Banach space (see \cite[Theorem 2.7]{MR1708448}), which is Hilbert provided that $(\X, \di, \m)$ is infinitesimally Hilbertian, in the sense of Definition \ref{def:infHil}.

\begin{definition}[Cheeger Energy]
We call \emph{Cheeger energy} the functional $\Ch \colon L^2(\X, \m) \to [0, + \infty]$ defined by setting for any $f \in L^2(\X, \m)$
\begin{equation}\label{eq:ChEn}
\Ch [f] := \big(|| f ||_{1, 2} - || f ||_{L^2(\X, \m)} \big)^2 = \inf_{\{g_i\}} \liminf_{i \to \infty} || g_i ||^2_{L^2 (\X, \m)}.
\end{equation}
\end{definition}

In order to give an explicit characterization of the Cheeger energy in terms of an integral of a local object, we introduce the notion of the \emph{minimal generalized upper gradient}. We say that a function $g \in L^2(\X, \m)$ is a \emph{generalized upper gradient} for $f \in L^2(\X, \m)$ if there exist two sequences $\{f_i\}, \{g_i\}_{i \in \N}$, such that:
\begin{enumerate}
    \item $f_i \to f$ and $g_i \to g$ in $L^2(\X, \m)$,
    \item $g_i$ is an upper gradient for $f_i$, for every $i \in \N$.
\end{enumerate}
The set of all the generalized upper gradients for a function $f \in L^2(\X, \m)$ is a closed convex subset of $L^2(\X, \m)$, which is in particular non-empty when $f \in H_{1, 2}(\X, \di, \m)$, as proved in \cite[Theorem 2.10]{MR1708448}. We denote by $|D f|_C$ the minimal generalized upper gradient, namely the unique element in $L^2(\X, \m)$ with the property that $|| f ||_{1, 2} = || f ||_{L^2(\X, \m)} + || | D f |_C ||_{L^2(\X, \m)}$. Hence, for any $f \in H_{1, 2}(\X, \di, \m)$, the Cheeger energy takes the form
\[
\Ch[f] = \int_\X | D f |_C^2 \, \de \m.
\]

\begin{remark}[Different definitions of Cheeger energy] 
This definition was first introduced by Cheeger in \cite{MR1708448} with the name of \emph{upper gradient 2-energy} of $f$ and can be compared with more recent definitions of energy introduced in the setting of metric measure spaces.

For example, Shanmugalingam introduced in \cite{MR1809341} the \emph{Newtonian space} $N^{1, 2}(\X, \di, \m)$, consisting of all the functions $f \colon \X \to \R$ such that $\int_\X f^2 \, \de \m$ and that
\begin{equation}\label{eq:new}
|f(\gamma_1) - f(\gamma_0)| \le \int_\gamma G
\end{equation}
holds for some $G \in L^2(\X, \m)$, out of a $\text{Mod}_2$-null set of curves. The potential $|D f|_S$ is then defined as the function $G$ in \eqref{eq:new} with minimal $L^2$-norm. In particular, Shanmugalingam proved also the connection between Newtonian spaces and Cheeger's functional: $f \in H_{1, 2}$ if and only if $\tilde f \in N^{1, 2}(\X, \di, \m)$ in the Lebesgue equivalence class of $f$, and the two notions of gradients $| D f|_S$ and $| D f |_C$ coincide $\m$-a.e. in $\X$. 

Another definition of energy was introduced in \cite{MR3152751}. In this case, the functional $\Ch_\ast$ is obtained using a relaxation procedure very similar to the one proposed by Cheeger, but the approximating functions $f_i$ are  now required to be Lipschitz and the upper gradients $g_i$ are replaced by the so-called \emph{relaxed gradients} of $f_i$ (see\cite[Definition 4.2]{MR3152751}). This functional can still be represented by the integration of a local object, denoted by $| D f |_\ast$, as
\[
\Ch_\ast[f]:= \dfrac12 \int_\X | D f|^2_\ast \, \de \m.
\]
The Sobolev space $W_\ast^{1, 2}(\X, \di, \m)$ is then defined as the domain of $\Ch_\ast$ endowed with the norm
\[
|| f ||_{W_\ast^{1, 2}} := \sqrt{|| f ||^2_{L^2(\X, \m)} + || | D f|_\ast ||^2_{L^2(\X, \m)}}.
\]
In particular, the space $(W_\ast^{1, 2}(\X, \di, \m), || \cdot ||_{W_\ast^{1, 2}})$ is complete. 

Directly from the definitions we have that for any $f \in L^2(\X, \m)$ it holds $|D f|_C \le | D f|_\ast$ $\m$-a.e., which in turn implies that the Cheeger energy $\Ch$ is smaller than the functional $\Ch_\ast$ and that
\[
W_\ast^{1, 2}(\X, \di, \m) \subseteq H_{1,2}(\X, \di, \m).
\]
However a series of remarkable results in \cite{MR3152751} (see specifically Theorem 6.2 and Theorem 6.3) guarantee that $\Ch = \Ch_\ast$ for any complete and separable metric spaces $(\X, \di)$ equipped with a locally finite measure.
\end{remark}

{The following two theorems show how the doubling condition and the validity of a Poincaré inequality ensure some nice properties related to the Cheeger energy and to the space $H_{1, 2}(\X, \di, \m)$.}

\begin{theorem}{\cite[Theorem 4.48]{MR1708448}} In a metric measure space $(\X,\di,\m)$ satisfying the doubling condition and supporting a Poincar\'e inequality, the norm on $H_{1, 2}(\X, \di, \m)$ is equivalent to a uniformly convex norm and, in particular, the space $H_{1, 2}(\X, \di, \m)$ is reflexive.
\end{theorem}
 
\begin{theorem}{\cite[Theorem 4.24, Theorem 5.1]{MR1708448}}\label{th:Cheeger}
In a metric measure space $(\X,\di,\m)$ satisfying the doubling condition and supporting a Poincar\'e inequality, 
 the subspace of all the Lipschitz functions is dense in $H_{1,2}(\X, \di, \m)$ and $|D f|_C(x) = \Lip[f](x)$ for $\m$-a.e. $x \in \X$.
 \end{theorem}

Finally we state the following result which is a consequence of \cite[Theorem 6.3]{MR3152751}:

\begin{prop}\label{prop:Chasrelax}
In a metric measure space $(\X,\di,\m)$, the Cheeger energy is equal to the relaxation in $L^2(\X,\m)$ of the functional
\begin{equation*}
    \begin{cases}
    \int \Lip[f]^2 \de \m \quad &\text{if }f\in \Lipfunc(\X)\\
    +\infty & \text{if }f\in L^2(\X,\m) \setminus \Lipfunc(\X),
    \end{cases}
\end{equation*}
that is 
\begin{equation*}
    \Ch [u] = \inf_{\Lipfunc(\X)\supset (f_n)\to u}\liminf_{n\to \infty} \int \Lip[f_n]^2 \de \m.
\end{equation*}
\end{prop}

\subsection{The Korevaar-Schoen Energy}\label{sec:ksenergy}

For every $p\in(1,\infty)$ and $r>0$, we define the $p$-energy density of size $r$ in the sense of Korevaar-Schoen as 
\begin{equation*}
    \ks_{p, r}[f](x):=\left|\fint_{B_{r}(x)} \frac{|f(y)- f(x)|^p}{r^{p}} \de \m(y)\right|^{1 / p},
\end{equation*}
for every $f\in \Lipfunc(\X)$. Accordingly, we introduce the $p$-Korevaar-Schoen energy of $f\in \Lipfunc(\X)$ at scale $r$ as
\begin{equation*}
    \Ks_{p,r}[f]:= \int \ks^p_{p, r}[f](x) \de \m(x).
\end{equation*}
Since the aim of this energy is to describe the gradient of a function or, more generally, an infinitesimal behaviour, it is natural to consider the limit as $r\to 0$ and define
\begin{equation}\label{eq:defKS}
    \Ks_p[f]:= \liminf_{r \to 0}  \Ks_{p,r}[f]
\end{equation}
We can then define an energy form $\tilde \Ks_p$ on $L^p(\X,\m)$ by taking the relaxation of $\Ks_p$, that is
\begin{equation*}
    \tilde\Ks_p[u] = \inf_{\Lipfunc \supset (f_n)\to u} \liminf_{n\to \infty} \Ks_p[f_n].
\end{equation*}

\begin{remark}[About this definition]
(i) The limit functional when the scale $r$ goes to $0$ is sometimes defined using the $\limsup$ instead of the $\liminf$ (see \eqref{eq:defKS}). In this work we decided to be consistent with the approach of (\cite{MR4316816}), which is our main reference. Nevertheless, since we are going to prove that the limit exists when $r$ tends to $0$, this is not an issue.\\
(ii) Usually, and in particular in (\cite{MR4316816}), the definition of the Korevaar-Schoen norm does not involve a relaxation procedure, while it is done by considering directly the potentials $\ks_{p,r}$ on $L^p(\X,\m)$. Once again this is not a big problem in this work, since in our context the two approaches coincide. This fact will be pointed out also in the following, when it is necessary.
\end{remark}
In the following we will only deal with the $2$-Korevaar-Schoen energy and in order to ease the notation we will drop the $2$ at the subscript, that is $\ks_r=\ks_{2,r}$, $\Ks_r=\Ks_{2,r}$, $\Ks=\Ks_2$ and $\tilde\Ks=\tilde\Ks_2$. The $2$-Korevaar-Schoen energy at positive scales $r$ is a quadratic form on $\Lipfunc(\X)$, as shown by the next result.

\begin{prop}\label{prop:KSrquad}
For every $r>0$, $\Ks_r$ is a quadratic form on $\Lipfunc(\X)$, that is 
\begin{equation*}
    \Ks_r[f+g] + \Ks_r[f-g]= 2 \Ks_r[f]+ 2 \Ks_r[g],
\end{equation*}
for every pair $f,g\in \Lipfunc(\X)$.
\end{prop}

\begin{proof}
This is actually a very easy consequence of the definition of the potentials $\ks_r$, since for every $f,g\in \Lipfunc(\X)$ it holds 
\begin{align*}
    \Ks_r[f+g] &+ \Ks_r[f-g]\\
    &= \int \fint_{B_{r}(x)} \frac{|f(y) + g(y)- f(x)-g(x)|^2}{r^{2}} +\frac{|f(y) - g(y)- f(x)+g(x)|^2}{r^{2}}  \de \m(y) \de \m (x)\\
    &= 2 \int \fint_{B_{r}(x)} \frac{|f(y) - f(x)|^2}{r^{2}} +\frac{|g(y)- g(x)|^2}{r^{2}}  \de \m(y) \de \m (x)\\
    &= 2 \Ks_r[f]+ 2 \Ks_r[g].
\end{align*}
This ensures that $\Ks_r$ is a quadratic form on $\Lipfunc(\X)$.
\end{proof}

\noindent In Section \ref{sec:convpot} we will use this last proposition to prove that also the Korevaar-Schoen energy $\tilde \Ks$ is a quadratic form, in particular it is the quadratic energy from which we will construct a distance which provides an infinitesimally Hilbertian structure on a locally Minkowski space. Our choice of the Korevaar-Schoen energy as reference quadratic energy is justified by a recent result proven by Gigli and Tyulenev. In \cite{MR4316816} (see Proposition 4.19 therein) they prove that on an RCD space the Korevaar-Schoen energy coincide, up to a dimensional constant, to the (quadratic) Cheeger energy.

\section{The intrinsic distance $\di_{\Ch}$}\label{sec:diCh}
In this short section we study the intrinsic distance associated to the Cheeger energy $\Ch$, which is defined as
\begin{equation}\label{eq:defdiCh}
    \di_\Ch(x,y)= \sup \{f(y)-f(x)\,:\, f\in H_{1,2}(\X,\di,\m) \cap\mathcal{C}(\X)\, \text{ and }\, |Df|_C\leq 1 \,\, \m-a.e.\},
\end{equation} 
where $\mathcal{C}(\X)$ denotes the set of continuous functions in $\X$. We point out that this is none other than the definition of intrinsic distance for a general Dirichlet form (see \cite{St95}), specialized to the Cheeger energy. In particular, the main purpose of this section is to prove that the distance $\di_{\Ch}$ is equivalent
to $\di$. To this aim, we follow the strategy developed by Ohta in \cite{MR2231926}, tailoring it to our setting. As a consequence of the equivalence $\di_\Ch\simeq \di$, we will prove that the supremum in \eqref{eq:defdiCh} can be taken over Lipschitz functions.

\begin{definition}[Maximal function]
For every measurable function $f$ and every radius $R>0$, we define the maximal function
\begin{equation*}
    M_R f(x)=\sup _{0<r<R} \fint_{B(x, r)}|f| \de \m.
\end{equation*}
\end{definition}

We are now going to prove a useful lemma, which shows how suitable maximal functions of an upper gradient $g$ of $f$ can provide an Haj\l asz-type estimate (cfr. \cite{MR1401074}) for $f$. This lemma is stated in a slightly different way and proven in \cite{MR1683160}, but we preferred to provide the proof anyway, in order to be self-contained and to avoid confusion to the reader.

\begin{lemma}\label{lem:maximal}
 Fix $R>0$ and let $(\X,\di,\m)$ be a metric measure space satisfying satisfying a doubling condition and supporting a Poincaré inequality, assume that $f$ is measurable and that $g$ is an upper gradient for $f$. Then there exists a constant $C:=C(C_D(R),C_P(2R))$ such that, if $x$ and $y$ are Lebesgue points of $f$ and $\di(x,y)<R$,
 it holds that
 \begin{equation*}
     |f(x)-f(y)| \leq C \di(x, y)\left(\left(M_{2 \Lambda \di(x, y)} g^{2}(x)\right)^{1 / 2}+\left(M_{2 \Lambda \di(x, y)} g^{2}(y)\right)^{1 / 2}\right).
 \end{equation*}
\end{lemma}

\begin{proof}
In this following we use the notation $C_D=C_D(R)$ and $C_P=C_P(2R)$.
For every $z\in \X$, $i\in \N$ we call $B_i(z):=B_{2^{-i}\di(x,y)}(z)$ and $\lambda B_i (z):=B_{2^{-i}\lambda \di(x,y)}(z) $ for every $\lambda>0$, since $x$ and $y$ are Lebesgue points of $f$ we have $f_{B_i}\to f(x)$ as $i$ goes to infinity. We can then perform the following estimate:
\begin{align*}
    |f(x)-f_{B_0(x)}| &\leq \sum_{i=0}^\infty |f_{B_{i+1}(x)}-f_{B_{i}(x)}| \leq \sum_{i=0}^\infty \fint_{B_{i+1}(x)} |f - f_{B_{i}(x)}| \de \m \\
    &= \sum_{i=0}^\infty \frac{1}{\m(B_{i+1})}\int_{B_{i+1}(x)} |f - f_{B_{i}(x)}| \de\m \leq C_D \sum_{i=0}^\infty \frac{1}{\m(B_{i})} \int_{B_{i}} |f - f_{B_{i}(x)}| \de\m \\
    & =  C_D \sum_{i=0}^\infty \fint_{B_{i}} |f - f_{B_{i}(x)}| \de\m \leq C_D C_P \cdot \di (x,y) \sum_{i=0}^\infty 2^{-i} \left( \fint_{\Lambda B_i(x) } g^2 \de \m \right) ^\frac 12,
\end{align*}
the third inequality follows from \eqref{eq:doubling}, while the last one is a consequence of \eqref{eq:poincare}.
Moreover, taking into account the definition of the maximal function, it is obvious that for every $i$ it holds
\begin{equation*}
    \fint_{\Lambda B_i(x) } g^2 \de \m \leq M_{ \Lambda \di(x, y)} g^{2}(x).
\end{equation*}
Then, putting together this last two estimate we can conclude that 
\begin{equation*}
    |f(x)-f_{B_0(x)}| \leq C_D C_P \cdot \di (x,y) \big(M_{ \Lambda \di(x, y)} g^{2}(x)\big)^{1/2}.
\end{equation*}
Proceeding in the same way we can also obtain that
\begin{equation*}
    |f(y)-f_{B_0(y)}| \leq C_D C_P \cdot \di (x,y) \big (M_{ \Lambda \di(x, y)} g^{2}(y)\big)^{1/2}.
\end{equation*}
On the other hand it can be easily notice that
\begin{align*}
    |f_{B_{0}(x)}-f_{B_{0}(y)}| & \leq|f_{B_{0}(x)}-f_{2 B_{0}(x)}|+|f_{B_{0}(y)}-f_{2 B_{0}(x)}|\\
    & \leq \fint_{B_{0}(x)}\left|f-f_{2 B_{0}(x)}\right| \de \m + \fint_{B_{0}(y)}\left|f-f_{2 B_{0}(x)}\right| \de \m \\
    & = \frac{1}{\m(B_{0}(x))}\int_{B_{0}(x)}\left|f-f_{2 B_{0}(x)}\right| \de \m + \frac{1}{\m(B_{0}(y))}\int_{B_{0}(y)}\left|f-f_{2 B_{0}(x)}\right| \de \m \\
    & \leq 2 C_D \fint_{2 B_{0}(x)}\left|f-f_{2 B_{0}(x)}\right| \de \m \\ & \leq 4 C_D C_P \cdot \di(x, y)\left(\fint_{2 \Lambda B_{0}(x)} g^{2} d \mu\right)^\frac 12 \\ & \leq 4 C_D C_P \cdot  \di(x, y)\left(M_{2 \Lambda \di(x, y)} g^{2}(x)\right)^{1 / 2} ,
\end{align*}
where we have applied \eqref{eq:doubling} and \eqref{eq:poincare} as before.
Thesis follows combining the last three inequalities, since by definition
\begin{equation*}
    \left(M_{2 \Lambda \di(x, y)} g^{2}(z)\right)^{1 / 2}\geq \left(M_{ \Lambda \di(x, y)} g^{2}(z)\right)^{1 / 2}\qquad \forall z\in \X.
\end{equation*}
\end{proof}

\begin{prop}\label{prop:equivCh}
 Let $(\X,\di,\m)$ be a length
 metric measure space satisfying the doubling condition \eqref{eq:doubling} and the Poincaré inequality \eqref{eq:poincare}, then we have
 \begin{equation*}
     \di \leq \di_\Ch \leq 2 C \cdot \di,
 \end{equation*}
 where $C$ is the constant introduced in Lemma \ref{lem:maximal}.
\end{prop}

\begin{proof}
Fix two points $x,y \in \X$ such that $\di(x,y)<R$ and a function $f\in \mathcal C(\X)$ such that $|Df|_C\leq 1$ $\m$-a.e. on $\X$. Since $f$ is continuous we can apply Lemma \ref{lem:maximal} to the points $x$ and $y$ and conclude that 
\begin{equation*}
    |f(x)-f(y)| \leq C \di(x, y)\left(\left(M_{2 \Lambda \di(x, y)} g^{2}(x)\right)^{1 / 2}+\left(M_{2 \Lambda \di(x, y)} g^{2}(y)\right)^{1 / 2}\right) \leq 2 C \di(x, y).
\end{equation*}
Taking into account the definition of the intrinsic distance $\di_\Ch$, this is sufficient to deduce that $\di_\Ch (x,y) \leq 2 C \cdot \di(x,y)$ for every pair of points such that $\di(x,y)<R$. The length space assumption on $(\X,\di,\m)$ allows to conclude the general inequality $\di_\Ch \leq 2 C \cdot \di$.

On the other hand, for every pair of points $x,y \in \X$, we can consider the function 
\begin{equation*}
    f (z):= \max \{ \di (x,y)- \di(x,z),0\}.
\end{equation*}
Notice that $f$ is $1$-Lipschitz thus $|Df|_C= \Lip[f]\leq 1$ a.e. (thanks to Theorem \ref{th:Cheeger}) and $f(x)-f(y)= \di(x,y)$, in particular this shows that $\di \leq \di_\Ch$.
\end{proof}

\noindent As an interesting corollary of this last result, we can provide an equivalent definition for the intrinsic distance $\di_\Ch$.

\begin{prop}\label{prop:dChLip}Under the assumptions of Proposition \ref{prop:equivCh}, in order to evaluate the distance $\di_\Ch$, we can restrict our attention to Lipschitz functions, that is:
\begin{equation}\label{eq:diChwithLip}
    \di_\Ch(x,y)= \sup \{f(y)-f(x)\,:\, f\in \Lipfunc(\X)\, \text{ and }\, \Lip[f]\leq 1 \,\, \m-a.e.\}.
\end{equation} 
\end{prop}

\begin{proof}
It is sufficient to prove that every $f\in H_{1,2}(\X,\di,\m) \cap\mathcal{C}(\X)$ such that $|Df|_C\leq 1$ $\m$-a.e. is a Lipschitz function, in fact in this case the condition $|Df|_C\leq 1$ becomes $\Lip[f]\leq 1$. It is obvious from the definition of $\di_\Ch$ that every function $f\in H_{1,2}(\X,\di,\m) \cap\mathcal{C}(\X)$ such that $|Df|_C\leq 1$ $\m$-a.e. is $1$-Lipschitz with respect to the distance $\di_\Ch$. On the other hand $\di$ and $\di_\Ch$ are equivalent distances according to Proposition \ref{prop:equivCh}, this allows to conclude the proof.
\end{proof}

\begin{remark}
In \cite{MR2231926} Ohta was actually able to prove that $\di= \di_\Ch$ whenever an additional rigidity condition, called \textit{strong doubling condition}, holds. In this paper we are not interested in adding this assumption to gain the equality of this two distances, because in the next section we will only prove that $\di_\Ch \simeq \di_\Ks$ and our main result (that is $\di \simeq \di_\Ks$) would not be improved.
\end{remark}

\section{Convergence of the Korevaar-Schoen Potentials}\label{sec:convpot}

This section is dedicated to prove the convergence of the Korevaar-Schoen potentials and to show some nice consequence of it. In particular, the locally Minkowski assumption on the reference metric measure space ensures the pointwise convergence of the potentials $\ks_r[f]$ as $r\to 0$, for every Lipschitz function $f\in \Lipfunc(\X)$ (Proposition \ref{prop:convofpotential}). As a consequence, we will get the convergence of the energies $\Ks_r$, proving also that $\tilde\Ks$ is a quadratic form on $L^2(\X,\m)$. Proposition \ref{prop:convofpotential} allows also to prove the equivalence of the energies $\Ks$ and $\Ch$ at the level of potentials, in view of the estimates provided by the following proposition. In the last part of the section we will prove a proper version of the Rademacher theorem (Proposition \ref{prop:differential}), finding, for $\m$-almost every regular point $x \in \mathcal R(\X)$, a linear function on $\R^{n(x)}$ representing the differential. 

\begin{prop}\label{prop:potestimate}
In a metric measure space $(\X,\di,\m)$ satisfying the doubling condition and supporting a Poincar\'e inequality, the following estimates for the Korevaar-Schoen potentials of $f\in\Lipfunc(\X)$ hold:
\begin{enumerate}
    \item $\limsup_{r\to 0}\ks^2_r[f](x)\leq \Lip[f](x)^2$ for every $x\in\X$,
    \item  $\liminf_{r \to 0}\ks^2_r[f](x)\geq \tilde C \Lip[f](x)^2$ for $\m$-a.e. $x\in\X$, for a suitable constant $\tilde C$ depending only on $C_D$.
\end{enumerate}
 
\end{prop}

\begin{proof}
1. This inequality can be easily deduced from the definition of the Korevaar-Schoen potentials, in fact:
\begin{equation}\label{eq:ksestimateup}
    \fint_{B_{r}(x)} \frac{|f(y)- f(x)|^2}{r^{2}} \de \m(y) \leq \left[\sup _{y\in B_r(x)} \frac{|f(y)-f(x)|}{r} \right]^2.
\end{equation}
Taking the $\limsup$ at both sides, we can conclude, keeping in mind \eqref{eq:pointlip}.

2. First of all we fix $f\in\Lipfunc(\X)$ and notice that $\Lip[f]$ is a measurable function, therefore it is possible to apply the Lebesgue differentiation theorem and find a Borel set $Z\subseteq X$ such that $\mu(X\setminus Z)=0$ and for every $\bar z\in Z$ it holds 
\begin{equation*}
   \lim_{s\to 0} \fint_{B_s(\bar z)} |\Lip[f](z)-\Lip[f](\bar z)|=0.
\end{equation*}
Then for every $\bar z\in Z$ and every $\delta >0$ the set 
\begin{equation*}
    E_\delta(\bar z):=\{x \in X \,:\, |\Lip[f](x)-\Lip[f](\bar z)|<\delta\}
\end{equation*}
has $\bar z$ as a density point, in fact
\begin{equation*}
    \frac{\m(B_s(\bar z)\setminus E_\delta(\bar z))}{\m(B_s(\bar z))}  \leq \frac 1 \delta\fint_{B_s(\bar z)} |\Lip[f](z)-\Lip[f](\bar z)| \to 0 \quad \text{as }s\to 0.
\end{equation*}
We are now going to prove that $\ks[f](x)\geq \tilde C \Lip[f](x)$ $\mu$-almost everywhere, for a suitable constant $\tilde C$ that will be explicitly determined during the proof. In particular we are going to show that, for every $\delta>0$ there exist a set $Z_1\subset Z$ such that $\m(Z_1 \setminus Z)<\delta$ and $\ks[f](x)\geq \tilde C \Lip[f](x)$ holds in $Z_1$. Once fixed an $\varepsilon>0$, we can find $Z_1\subset Z$ and $r_1>0$ such that $\mu(Z\setminus Z_1)<\delta$ and for every $z_1\in Z_1$ it holds 
\begin{equation}\label{eq:defZ1}
    \frac{|f(z_1)-f(z)|}{\di (z_1,z)} < (\Lip f(z_1)+ \varepsilon) \qquad \forall z \in \X \text{ such that }\di(z_1,z)\leq r_1.
\end{equation}
Let $\bar z\in Z_1$ be a density point of the set $Z_1$. Then, up to an $\m$-negligible restriction of $Z_1$, for every $0<r<r_1\leq R$ small enough, there exists $z_r\in \X$ such that $\di(z_r,\bar z)=\frac 34 r$ and
\begin{equation}\label{eq:lipzr}
    |f(z_r)-f(\bar z)| \geq \frac 34 r \cdot (\Lip(\bar z)-\varepsilon),
\end{equation}
this is a quite obvious consequence of Proposition \ref{prop:pointlip}.
Then for $r$ small enough we can find $z'_r\in B_{\frac{1}{8}r}(z_r)\cap Z_1 \cap E_\varepsilon(\bar z)$, in fact for every $0<r<r_1$ it holds that $B_r(\bar z) \subset B_{2r}(z_r)$ and then
\begin{equation*}
    \m \big(B_{\frac{1}{8}r}(z_r)\big) \geq \frac{1}{C_D(R)^5} \m(B_{2r}(z_r)) \geq \frac{1}{C_D(R)^5} \m (B_r(\bar z)).
\end{equation*}
In particular we have that definitely 
\begin{equation*}
    \frac{\m \big(B_{\frac{1}{8}r}(z_r)\big)}{\m (B_r(\bar z))} \geq \frac{1}{C_D(R)^5} > 0,
\end{equation*}
while, on the other hand, $Z_1 \cap E_\varepsilon(\bar z)$ has a density point in $\bar z$.
Then, according to \eqref{eq:defZ1}, we can deduce that
\begin{equation*}
    |f(x)-f(z'_r)|\leq (\Lip[f](\bar z) + 2\varepsilon) \di(x,z'_r)
\end{equation*}
for every $x\in B_{r_1}(z'_r)$. Consequently, combining this last inequality with \eqref{eq:lipzr}, we can deduce that for every $z\in B_{\frac 18 r}(z_r)$
\begin{align*}
    |f(z)-f(\bar z)| &\geq |f(z_r)-f(\bar z)| - |f(z_r)-f(z'_r)| - |f(z'_r)-f(z)|\\
    & \geq \frac 34 r \cdot (\Lip[f](\bar z)-\varepsilon) - \frac 38 r \cdot (\Lip[f](\bar z) + 2\varepsilon) = \frac 38 r \cdot \Lip[f](\bar z) - \frac 32 \varepsilon r.
\end{align*}
Putting everything together we obtain 
\begin{align*}
    \fint_{B_r(\bar x)} \frac{|f(z)-f(\bar z)|^2}{r^2} \de \m  (z)&=\frac{1}{\m(B_r(\bar x))}\int_{B_r(\bar x)} \frac{|f(z)-f(\bar z)|^2}{r^2} \de \m (z)\\
    &\geq \frac{1}{C_D(R)^5}\cdot \frac{1}{\m(B_{\frac 18 r}(\bar x))}\int_{B_{\frac18 r}(\bar x)} \frac{|f(z)-f(\bar z)|^2}{r^2} \de \m (z)\\
    & \geq \frac{1}{C_D(R)^5} \left[  \frac 38 \Lip[f](\bar z) - \frac 32 \varepsilon  \right]^2.
\end{align*}
Since we can perform this estimate for every $r$ small enough we can conclude that for every $\bar z \in Z_1$
\begin{equation*}
    \liminf_{r\to 0} \ks^2_r[f](\bar z)\geq \frac{1}{C_D(R)^5} \left[  \frac 38 \Lip[f](\bar z) - \frac 32 \varepsilon  \right]^2.
\end{equation*}
This inequality is true for every $\varepsilon>0$, then we can deduce that 
\begin{equation*}
    \ks^2_r[f](\bar z) \geq \frac{1}{C_D(R)^5} \left[  \frac 38\right]^2 \Lip[f](\bar z)^2,
\end{equation*}
this concludes the proof.
\end{proof}

\begin{remark}\label{rmk:unifboundks}
Notice that, the inequality \eqref{eq:ksestimateup} also allows to conclude that $\ks_r[f](x)\leq \Lipglob[f]$ for every $x\in \X$ and every $r>0$, this global and uniform bound will be useful in the following.
\end{remark}

\noindent We point out that for Proposition \ref{prop:potestimate} (and for Remark \ref{rmk:unifboundks}) the locally Minkowski assumption is not needed, but from now on this hypothesis will play a fundamental role. 

Moreover, in the following we will always work with metric measure spaces which satisfy the doubling condition and support a Poincaré inequality.\\

We can now prove the convergence of the Korevaar-Schoen potentials.

\begin{prop}\label{prop:convofpotential}
Let $(\X,\di,\m)$ be a locally Minkowski space and let $f:\X\to \R$ be a Lipschitz function. Then for $\m$-almost every $x\in \X$ the limit
\begin{equation}\label{eq:limitks}
    \ks[f](x):=\lim_{r\to 0^+} \ks_r[f](x)
\end{equation}
exists and it is finite.
\end{prop}

\begin{proof}
Once fixed $m\in \N$, for every $x\in \mathcal{R}(f)$ we can find $r(x)$ such that $i^{x,r(x)}$ satisfies 2 and 3 in Definition \ref{def:locMink} with $\varepsilon = \varepsilon_m:= \frac 1m$. Then $\{B_{r(x)}(x)\}_{x\in \mathcal{R}(f)}$ is an open cover of $ \mathcal{R}(f)$ and since $(\X,\di)$ is hereditarily Lindel\"of (because it is separable) there exists a countable subcover. Now consider an element of the subcover $B_{r(x)}(x)$ and notice that, since $i^{x,r(x)}$ is bi-Lipschitz, it is invertible on its image and the map
\begin{equation*}
    f \circ (i^{x,r(x)})^{-1}: (\R^{n(x)},\norm{\cdot}) \supset i^{x,r(x)}(B_{r(x)}(x)) \to \R
\end{equation*}
is Lipschitz. The standard Euclidean norm $\norm{\cdot}_{eu}$ is equivalent to $\norm{\cdot}$ and thus $f \circ (i^{x,r(x)})^{-1}$ is Lipschitz also with respect to $\norm{\cdot}_{eu}$. In particular we can apply the Rademacher theorem and conclude that $f \circ (i^{x,r(x)})^{-1}$ is differentiable $\Leb^{n(x)}$-almost everywhere on $i^{x,r(x)}(B_{r(x)}(x))$. Taking into account property 3 in Definition \ref{def:locMink}, we can actually conclude that for $\m$-almost every point $y$ in $B_{r(x)}(x)$ the function $i^{x,r(x)}$ is differentiable in $i^{x,r(x)}(y)$. Putting together what we have up to now, we can deduce that there exists a $\m$-full measure set $A$ such that for every $x\in A$ and $m\in \N$ we can find an $\varepsilon_m$-bi-Lipschitz map 
\begin{equation*}
    i_m: A_m \to (\R^{n_m}, \norm{\cdot})
\end{equation*}
where $A_m$ is an open neighborhood of $x$ and $(\R^{n_m}, \norm{\cdot})$ is a suitable normed spaces (with $n_m \leq N$), such that $f \circ (i_m)^{-1}$ is differentiable in $i_m(x)$. 
We are going to prove that the limit \eqref{eq:limitks} exists for every $x\in A$. To this end, we fix $m \in \N$ and call $z=i_m(x)$ and $g=f \circ (i_m)^{-1}$. Now $g$ is differentiable in $z$, with differential $\nabla g(z)$, that is
\begin{equation*}
    g(z')-g(z) = \scal{\nabla g (z)}{z'-z} + o (\norm{z'-z}).
\end{equation*}
In particular, we can find $\bar r >0$ such that $B_{\bar r}(x)$ is inside the domain of $i_m$ and 
\begin{equation*}
    | g(z')-g(z) - \scal{\nabla g (z)}{z'-z}| \leq \varepsilon_m  (\norm{z'-z}),
\end{equation*}
whenever $z'\in i_m(B_{\bar r}(x))$. We can then take any $r<\bar r$ and perform the following computation
\begin{equation}\label{eq:estimateinchart}
    \begin{split}
        \ks_r^2[f](x) &= \fint_{B_r(x)} \frac{|f(x')-f(x)|^2}{r^2} \de \m(x')\\
        &= \frac{1}{\m(B_r(x))}\int_{B_r(x)} \frac{|f(x')-f(x)|^2}{r^2} \de \m(x') \\
    &=\frac{1}{\m(B_r(x))} \int_{i_m(B_r(x))} \frac{|g(z')-g(z)|^2}{r^2} \de [(i_m)_\# \m] (z') .
    \end{split}
\end{equation}
On the other hand, since $i_m$ is $\varepsilon_m$-bi-Lipschitz, it holds that
\begin{equation*}
    B_{(1-\varepsilon_m) r} (z) \subseteq i_m (B_r(x)) \subseteq  B_{(1+\varepsilon_m) r} (z)
\end{equation*}
and consequently that, for some a constant $c$ depending on $i_m$,
\begin{align*}
    (1-\varepsilon_m)c \cdot \Leb^{n_m} (B_{(1-\varepsilon_m) r} (z)) &\leq (i_m)_\# \m (B_{(1-\varepsilon_m) r} (z)) \leq \m(B_r(x)) \\
    &\leq (i_m)_\# \m (B_{(1+\varepsilon_m) r} (z)) \leq (1+\varepsilon_m)c \cdot \Leb^{n_m} (B_{(1+\varepsilon_m) r} (z)).
\end{align*}
As a result, we can deduce that 
\begin{equation*}
    c \cdot \Leb^{n_m} (B_{(1-\varepsilon_m) r} (z)) = \bigg[\frac{1-\varepsilon_m}{1+\varepsilon_m}\bigg]^{n_m} c \cdot \Leb^{n_m} (B_{(1+\varepsilon_m) r} (z)) \geq \frac{(1-\varepsilon_m)^{n_m}}{(1+\varepsilon_m)^{n_m+1}} \m(B_r(x))
\end{equation*}
Then, taking into account \eqref{eq:estimateinchart}, we can conclude that 
\begin{align*}
    \ks_r^2[f](x) &\geq \frac{1}{\m(B_r(x))} (1-\varepsilon_m) \int_{B_{(1-\varepsilon_m) r}(z)} \frac{|g(z')-g(z)|^2}{r^2}  \de [c\cdot \Leb^{n_m}] (z') \\
    &\geq \bigg[\frac{1-\varepsilon_m}{1+\varepsilon_m}\bigg]^{N+1} \fint_{B_{(1-\varepsilon_m) r}(z)} \frac{|g(z')-g(z)|^2}{r^2}  \de \Leb^{n_m} (z') \\
    &\geq \bigg[\frac{1-\varepsilon_m}{1+\varepsilon_m}\bigg]^{N+1} \left[ \fint_{B_{(1-\varepsilon_m) r}(z)} \frac{|\scal{\nabla g(z)}{z'-z}|^2}{r^2} \de  \Leb^{n_m} (z') - \varepsilon_m\right]\\
    &= \bigg[\frac{1-\varepsilon_m}{1+\varepsilon_m}\bigg]^{N+1} \left[ (1-\varepsilon_m)^2\fint_{B_{1}(z)} |\scal{\nabla g}{z'-z}|^2 \de  \Leb^{n_m} (z') - \varepsilon_m\right].
\end{align*}
With an analogous argument we can deduce that 
\begin{equation*}
    \ks_r^2[f](x) \leq \bigg[\frac{1+\varepsilon_m}{1-\varepsilon_m}\bigg]^{N+1} \left[ (1+\varepsilon_m)^2\fint_{B_{1}(z)} \frac{|\scal{\nabla g}{z'-z}|^2}{r^2} \de [c\cdot \Leb^{n_m}] (z') + \varepsilon_m\right].
\end{equation*}
On the other hand, it is easy to notice that, since $\ks_r^2[f]$ is uniformly bounded by $\Lipglob[f]$ (see Remark \ref{rmk:unifboundks}), the quantity 
\begin{equation*}
    \fint_{B_{1}(z)} \frac{|\scal{\nabla g}{z'-z}|^2}{r^2} \de [c\cdot \Leb^{n_m}] (z')
\end{equation*}
can be itself bounded uniformly. This observation, combined with the estimates, allows to prove that
\begin{equation*}
   | \ks_r[f](x)- \ks_{r'}[f](x)| \leq O(\varepsilon_m)
\end{equation*}
for every $r,r'< \bar r$. This is sufficient to conclude the proof.
\end{proof}

\begin{corollary}\label{cor:convKS}
For every Lipschitz function $f:\X\to \R$ in locally Minkowski space $(\X,\di,\m)$ it holds that
\begin{equation}\label{eq:convKS}
    \Ks[f]= \int \ks^2[f] \de \m = \lim_{r\to 0} \int \ks^2_r[f]  \de \m = \lim_{r \to 0} \Ks_r[f]
\end{equation} and $\Ks$ is a quadratic form on $\Lipfunc(\X)$.
Moreover we have that 
\begin{equation*}
    \tilde  C  \Lip[f](x)^2  \leq \ks^2[f](x)\leq \Lip[f](x)^2  \qquad \text{for $\m$-almost every }x\in \X,
\end{equation*}
where $\tilde C$ is the constant appearing in Proposition \ref{prop:potestimate}, and consequently $\tilde C\Ch[f]\leq \Ks[f] \leq \Ch[f]$.
\end{corollary}

\begin{proof}
Relation \eqref{eq:convKS} is an easy consequence of the dominated convergence theorem (keep in mind Remark \ref{rmk:unifboundks}). Then $\Ks$ is a quadratic form on $\Lipfunc(\X)$, being the limit of the quadratic forms $\Ks_r$ (see Proposition \ref{prop:KSrquad}). The second part of the statement is the natural combination Proposition \ref{prop:convofpotential} and Proposition \ref{prop:potestimate}.
\end{proof}

\begin{corollary}\label{cor:ksquadr}
In locally Minkowski space $(\X,\di,\m)$, $\tilde \Ks$ is a quadratic form on $L^2(\X,\m)$, that is
\begin{equation*}
    \tilde\Ks[f+g] + \tilde\Ks[f-g]= 2 \tilde\Ks[f]+ 2 \tilde\Ks[g],
\end{equation*}
for every $f,g \in L^2(\X,\m)$.
\end{corollary}

\begin{proof}
We prove that 
\begin{equation*}
    \tilde\Ks[f+g] + \tilde\Ks[f-g]\leq 2 \tilde\Ks[f]+ 2 \tilde\Ks[g] \qquad \forall f,g \in L^2(\X,\m),
\end{equation*}
the opposite inequality can be proven analogously. We can then assume that $\tilde\Ks[f], \tilde\Ks[g] <\infty$, thus there exist two sequences $(f_n)_{n\in \N}\subset\Lipfunc(\X)$ and $(g_n)_{n\in \N}\subset\Lipfunc(\X)$, converging in $L^2(\X,\m)$ to $f$ and $g$ respectively, such that $\Ks[f_n]\to\tilde\Ks[f]$ and $\Ks[g_n]\to\tilde\Ks[g]$. Moreover, the sequences $(f_n+g_n)_{n\in \N},(f_n-g_n)_{n\in \N}\subset\Lipfunc(\X)$ converge in $L^2(\X,\m)$ to $f+g$ and $f-g$ respectively, therefore
\begin{equation*}
\begin{split}
    2 \tilde\Ks[f]+ 2 \tilde\Ks[g] = \liminf_{n\to 0} 2 \Ks[f_n]+ 2 \Ks[g_n] &= \liminf_{n\to 0} \Ks[f_n+g_n] + \Ks[f_n-g_n] \\
    &\geq \tilde\Ks[f+g] + \tilde\Ks[f-g],
\end{split}
\end{equation*}
this concludes the proof.
\end{proof}

The convergence of the Korevaar-Schoen potentials also ensures the convexity of the potential $\ks$, in the form of the next corollary. This result is important especially because it allows to prove a version of Mazur's lemma adapted to our setting, i.e. Lemma \ref{lem:mazur}.

\begin{corollary}\label{cor:ksasnorm}
Let $f$ be a finite convex combination of Lipschitz functions, that is $ f= \sum_{i=1}^{N} \lambda_i f_i$ with $f_i \in \Lipfunc(\X)$ for every $i$ and $\sum_{i=1}^{N}\lambda_1=1$. Then, it holds that 
\begin{equation}\label{eq:ksnorm}
    \ks[f] \leq \sum_{i=1}^{N} \lambda_i \ks[f_i],
\end{equation}
$\m$-almost everywhere.
\end{corollary}

\begin{proof}
We prove this result only in the case where $N=2$, the general case can be done in the same way. We are going to prove \eqref{eq:ksnorm} for every $x\in \X$ such that 
\begin{equation*}
     \ks[f_1](x)=\lim_{r\to 0^+} \ks_r[f_1](x), \qquad \ks[f_2](x)=\lim_{r\to 0^+} \ks_r[f_2](x) \qquad \text{and} \qquad  \ks[f](x)=\lim_{r\to 0^+} \ks_r[f](x),
\end{equation*}
notice that this set has $\m$-full measure, accordingly to Proposition \ref{prop:convofpotential}. For such an $x$ we can make the following computation
\begin{align*}
    \ks^2[f](x) &= \lim_{r\to 0}  \fint_{B_{r}(x)} \frac{|f(y)- f(x)|^2}{r^{2}} \de \m(y) \\
    &= \lim_{r\to 0}  \fint_{B_{r}(x)} \frac{|\lambda_1 (f_1(y)-f_1(x))+ \lambda_2 (f_2(y)-f_2(x))|^2}{r^2} \de \m(y) \\
    &\leq \lim_{r\to 0} \left[ \lambda_1 \left( \fint_{B_{r}(x)} \frac{|f_1(y)- f_1(x)|^2}{r^{2}} \de \m(y)\right)^\frac 12 + \lambda_2 \left( \fint_{B_{r}(x)} \frac{|f_2(y)- f_2(x)|^2}{r^{2}} \de \m(y)\right)^\frac 12   \right]^2\\
    &=\lim_{r \to 0} \big[\lambda_1\ks_r[f_1](x)+ \lambda_2 \ks_r[f_2](x)\big]^2\\
    &= \big[\lambda_1\ks[f_1](x)+ \lambda_2 \ks[f_2](x)\big]^2
\end{align*}
where the $\leq$ is a consequence of the Cauchy-Schwartz inequality. Taking the square root from both sides, we obtain the desired inequality.
\end{proof}

\begin{lemma} \label{lem:mazur}
   Given a Lipschitz function $f\in \Lipfunc(\X)$ and two sequences $\{f_i\}_{i\in \N}\subset \Lipfunc(\X)$ and $\{g_i\}_{i\in \N} \subset L^2(\X,\m)$ such that 
   \begin{equation*}
       f_i \to f \text{ in }L^2(\X,\m) \qquad \text{and} \qquad \ks[f_i] \leq g_i \leq K \,\,\,\, \m-a.e.,
   \end{equation*}
   for a fixed constant $K$, then 
   \begin{equation*}
       \ks[f](x)\leq \limsup_{i\to \infty} g_i(x),
   \end{equation*}
   for $\m$-almost every $x\in \X$.
\end{lemma}

\begin{proof}
Notice that the sequence $\{f_i\}_{i\in \N}\subset \Lipfunc(\X)\subset H_{1,2}(\X,\di,\m)$ is bounded in $H_{1,2}(\X,\di,\m)$, in fact Corollary \ref{cor:convKS} ensures that for every $i$
\begin{equation*}
     \Lip[f_i] \leq \frac{1}{\tilde C} \ks[f_i] \leq \frac K {\tilde C},
\end{equation*}
$\m$-almost everywhere. Moreover, since the Banach space $H_{1,2}(\X,\di,\m)$ is reflexive,  $\{f_i\}_{i\in \N}$ weakly converges (up to subsequences) to $f$. We can then apply a well known variant of Mazur's lemma (see \cite[Exercise 3.4]{MR2759829}) and deduce the existence of non-negative coefficients $\{a_{n,i}\}_{i\geq n}$ such that $\sum_{i=n}^{\infty} a_{n,i} =1$ and the functions $$\hat f_n= \sum_{i=n}^{\infty} a_{n,i} f_i$$ converge to $f$ strongly in $H_{1,2}(\X,\di,\m)$. Notice that, applying Corollary \ref{cor:ksasnorm} and keeping in mind the uniform bound on the functions $g_i$, we can deduce 
\begin{equation*}
    \ks[\hat f_n](x) \leq \sum_{i=n}^{\infty} a_{n,i} \ks[f_i](x) \leq \sum_{i=n}^{\infty} a_{n,i} g_i(x) \leq \sum_{i=n}^{\infty} a_{n,i} \cdot  \sup_{i\geq n} g_i(x)=  \sup_{i\geq n} g_i(x)
\end{equation*}
for $\m$-almost every $x\in \X$. On the other hand, since $\hat f_n \to f$ in $H_{1,2}(\X,\di,\m)$, it holds that $\norm{\Lip(f-\hat f_n)}_{L^2}\to 0$, then Corollary \ref{cor:convKS} yields that $\norm{\ks(f-\hat f_n)}_{L^2}\to 0$. As a consequence, up to possibly pass to a subsequence, we deduce that $\ks(f-\hat f_n)\to 0$, $\m$-almost everywhere. Therefore, using Corollary \ref{cor:ksasnorm} once again, we conclude
\begin{equation*}
    \ks[f](x) \leq \lim_{n \to \infty} \big[ \ks[f-\hat f_n](x) + \ks[\hat f_n](x) \big] \leq \lim_{n \to \infty} \big[ \ks[f-\hat f_n](x) + \sup_{i\geq n} g_i(x) \big] = \limsup_{i\to \infty} g_i(x),
\end{equation*}
for $\m$-almost every $x\in \X$.
\end{proof}

We are now going to prove the Rademacher theorem for locally Minkowski metric measure spaces. As a consequence, we will obtain an explicit form for the Korevaar-Schoen potential $\ks[f](x)$, for $\m$-almost every $x\in \X$, in term of the differential of $f$ in the point $x$ (Corollary \ref{cor:kswithdiff}).

\begin{prop}\label{prop:differential}
Let $(\X,\di,\m)$ be a locally Minkowski space and let $f:\X\to \R$ be a Lipschitz function. Then for $\m$-almost every $x \in \X$ we can find a sequence of radii $\{r_m\}_{m\geq 1} \to 0$ and a sequence of maps
\begin{equation*}
    i_m \colon (B_{2 r_m}(x),\di_{ r_m}) \to (\R^{n(x)}, \norm{\cdot})
\end{equation*}
satisfying for every $m \geq 1$ the following properties, where $\varepsilon_m=\frac 1m$:
\begin{enumerate}
    \item[1)] $i_m$ is a $\varepsilon_m$-bi-Lipschitz map with $i_m(x)=0$;
    \item[2)] $B_{2(1-\varepsilon_m)} (0)\subseteq i_m (B_{r_m}(x))$;
    \item[3)] we have that \begin{equation*}
        (1-\varepsilon_m)c(x) \cdot \Leb^{n(x)}\leq (i_m)_\# \bigg[\frac \m{\m(B_{r_m}(x))}\bigg] \leq (1+\varepsilon_m)c(x) \cdot \Leb^{n(x)},
    \end{equation*}
    on the set $i_m (B_{2r_m}(x))$;
    \item[4)] the sequence $g_m := \frac{(f-f(x))\circ (i_m)^{-1}}{r_m}$ converges uniformly on $B_1(0)$ to the linear function $g$ given by
\begin{equation*}
    g\colon(\R^{n(x)}, \norm{\cdot}) \to \R, \qquad g(z)= \scal{v_g}{z}.
\end{equation*}
\end{enumerate}
\end{prop}

\noindent Before presenting the proof of this result, we recall some preliminary notions used therein. In fact, the argument showing the validity of point 4) in Proposition \ref{prop:differential} closely follows some of the ideas presented in \cite{MR1708448}, as we specify now.\\

\underline{\emph{Notions and results required in the proof of Proposition \ref{prop:differential}.}}
We start by introducing some  useful definitions in the setting of metric spaces $(\X, \di)$. A curve $\gamma \colon [0, \infty) \to \X$ is said to be a \emph{half line} if it holds
\[
\di(\gamma_t, \gamma_s) = |t - s|, \text{ for any } t, s \ge 0,
\]
while a curve $\gamma \colon \R \to \X$ is called a \emph{line} if it holds
\[
\di(\gamma_t, \gamma_s) = |t - s|, \text{ for any } t, s \in \R.
\]

\noindent To a half line $\gamma$ we can associate the \emph{Busemann function} ${\rm b} \colon \X \to \R$ defined by setting
\[
{\rm b} (x) := \inf_{s \ge 0} b_{\gamma, s}(x) = \lim_{s \to +\infty} b_{\gamma, s}(x),
\]
where $b_{\gamma, s}(x) := \di (x, \gamma(s)) - s$, being $s$ the arc-length parameter. We remark that this function is actually well-defined, since the triangle inequality ensures that the functions $b_{\gamma, s}(x)$ are uniformly bounded below on any compact subset and $b_{\gamma, s_2} \le b_{\gamma, s_1}$ if $s_1 \le s_2$. In particular, the Busemann function can be equivalently defined using the infimum on $s \ge 0$ or the limit as $s \to +\infty$. 

A similar construction allows us to associate two Busemann functions $\rm b^+, \rm b^-$ to a line $\gamma \colon \R \to \X$ by posing
\[
\begin{split}
{\rm b}^+(x) &:= \inf_{s \ge 0} b_{\underline{\gamma}, s}(x) = \lim_{s \to +\infty} b_{\underline{\gamma}, s}(x), \\
{\rm b}^-(x) &:= \inf_{s \ge 0} b_{-\underline{\gamma}, s}(x) = \lim_{s \to +\infty} b_{-\underline{\gamma}, s}(x),
\end{split}
\]
where $\underline\gamma = \gamma|_{[0, \infty)}$ is the half line associated to $\gamma$ and we define $-\underline\gamma \colon [0, \infty) \to \X$ as $-\underline\gamma(s) := \gamma(-s)$. A direct application of the triangular inequality implies that
\[{\rm b}^+ + {\rm b}^- \ge 0.\]

Another crucial notion we will use in the following is the one of \emph{generalized linear functions}:

{\begin{definition}{\cite[Definition 7.1, Definition 8.1]{MR1708448}} A Lipschitz function $f \in \Lipfunc(\X)$ is said to be
\begin{itemize} 
\item[a)]   \emph{harmonic} if for every bounded open set $U\subset \X$ and every $h \in H_{1, 2}(\X,\di,\m)$ with $\text{supp}(h) \subset \subset U$, it holds that
\begin{equation*}
     \| |D(f + h)|_C \|_{L^2} \ge \| |D f|_C \|_{L^2} = \| \Lip[f] \|_{L^2},
\end{equation*}
\item[b)] \emph{generalized linear} if 
\begin{itemize}
\item[(i)] either $f \equiv 0$ or the range of $f$ is $(-\infty, +\infty)$;
\item[(ii)] $f$ is harmonic;
\item[(iii)] $\Lip[f] \equiv c$ for some $c \in \R$.
\end{itemize}
\end{itemize}
\end{definition}}

In the setting of metric spaces $(\X, \di)$ equipped with a doubling measure $\m$, in the sense of \eqref{eq:doubling}, and supporting a Poincar\'e inequality \eqref{eq:poincare}, the class of generalized linear functions satisfies some remarkable properties, as deeply studied in \cite[Chapter 8]{MR1708448}. In particular, in this framework, a first useful result consists in \cite[Theorem 8.10]{MR1708448}, which states that if $f \colon \X \to \R$ is a generalized linear function, then for any $\bar x \in \X$ there exists a line $\gamma \colon \R \to \X $ with $\gamma(0) = \bar x$ and with the property that $\gamma$ is an \emph{integral curve} for $| D f|_C = \Lip(f)$. A further result in this direction is given by \cite[Theorem 8.11]{MR1708448} and it ensures the validity of the following chain of inequalities
\begin{equation}\label{eq:genlin}
f(\bar x) - \Lip(f) \cdot {\rm b}^+ \le f \le f(\bar x) + \Lip(f) \cdot {\rm b}^-,
\end{equation}
where $\gamma$ is the line provided by \cite[Theorem 8.10]{MR1708448}. 

{\begin{lemma}\label{lem:lin}
Let $\R^n$ be equipped with a $C^1$-norm $|| \cdot ||$ and $f \colon \X \to \R$ be a generalized linear function. Then it holds that
\begin{equation}\label{eq:linear}
f(\bar x) - \Lip(f) \cdot {\rm b}^+ = f = f(\bar x) + \Lip(f) \cdot {\rm b}^-,
\end{equation}
where ${\rm b}^+$ and ${\rm b}^-$ are the Busemann functions associated to the line $\gamma$ with $\gamma(0)=0$ provided by \cite[Theorem 8.10]{MR1708448}. Moreover, $f$ is a linear function.
\end{lemma}}

\begin{proof}
{ In view of \eqref{eq:genlin},} the result follows if we prove that ${\rm b}^+ =- {\rm b}^-$ {and that ${\rm b}^+$ is a linear function}. In order to prove it, we first notice that every line $\gamma$ in the space $(\R^n,\norm{\cdot})$ with $\gamma(0)=0$ is of the type $\gamma(s)=sv$, where $v\in \R^n$ is a unit vector, that is $\norm{v}=1$. In particular, we can compute the Busemann function ${\rm b}^+$:
\begin{equation*}
    \begin{split}
        {\rm b}^+(x) = \lim_{s \to +\infty} b_{\underline{\gamma}, s}(x) &= \lim_{s \to +\infty} \norm{sv-x}- \norm{sv}= \lim_{s \to +\infty} s \cdot [\norm{v-x/s}- \norm{v}] =\\
        &=\lim_{t \to 0} \frac{\norm{v-tx}- \norm{v}}{t} = -\scal{\nabla \norm{\cdot}(v)}{x}.
    \end{split}
\end{equation*}
With the analogous computation for ${\rm b}^-$, we show that ${\rm b}^-(x) = \scal{\nabla \norm{\cdot}(v)}{x}$, concluding the proof.
\end{proof}


\begin{proof}[Proof of Proposition \ref{prop:differential}]
Let $x \in \mathcal R (\X)$ be a fixed regular point. We recall that the locally Minkowski property ensures that for every $\varepsilon > 0$ we can find a radius $r(\varepsilon) > 0$ for which the map $i^{x, r} \colon (B_r(x), \di_r) \to (\R^{n(x)}, \norm{\cdot})$ is $\varepsilon$-bi-Lipschitz for any $r < r(\varepsilon)$ and for which the condition on the pushforward measure of $\m^{x, r}$ through the map $i^{x, r}$ expressed in \eqref{eq:pfmea} holds true. Hence, for any $m \in \N$, $m \ge 1$, we take $\tilde r_m<r(\varepsilon_m)$ and we set $r_m := \frac{\tilde r_m}{m+1}$. Then, for every $m\geq 1$, we consider the map 
\begin{equation*}
    \tilde i _m : (B_{\tilde r_m}(x), (m+1)\di_{ \tilde r_m}) =(B_{\tilde r_m}(x),\di_{ r_m})  \to (\R^{n(x)}, \norm{\cdot})
\end{equation*}
defined as $\tilde i_m (y) = (m+1)\cdot i^{x, \tilde r_m}$. The properties of $i^{x, \tilde r_m}$ transfer to its rescaling $\tilde i_m $, in particular $\tilde i_m (x)=0$, $B_{m}(0) \subseteq \tilde i_m(B_{\tilde r_m }(x))$, $\tilde i_m$ is $\varepsilon_m$-bi-Lipschitz and 
\begin{equation*}
    (1-\varepsilon_m)c(x) \cdot \Leb^{n(x)}\leq \big(\tilde i_m\big)_\# \frac{\m}{\m(B_{r_m}(x))} \leq (1+\varepsilon_m)c(x) \cdot \Leb^{n(x)}, \quad \text{on }\tilde i_m (B_{\tilde r_m}(x)).
\end{equation*}
Notice that the functions 
\begin{equation*}
    g_m := \frac{(f-f(x))\circ (\tilde i_m)^{-1}}{r_m}
\end{equation*}
are equi-Lipschitz, in fact for every $a,b \in \tilde i_m(B_{\tilde r_m}(x))$ it holds that
\begin{equation*}
    \begin{split}
        |g_m(a) - g_m(b)| &= \bigg| \dfrac{f( \tilde i_m^{-1}(a))- f(\tilde i_m^{-1}(b))}{r_m} \bigg| \leq \Lip(f) \dfrac{\di\big(\tilde i_m^{-1}(a), \tilde i_m^{-1}(0)\big)}{r_m}\\
        &= \Lip(f)\di_{r_m}\big(i_m^{-1}(a), i_m^{-1}(0)\big) \leq \Lip(f) (1+ \varepsilon_m) \norm{a-b}\leq 2 \Lip(f) \norm{a-b}.
    \end{split}
\end{equation*}
Observe moreover that, for every fixed $\bar m>1$, the function $g_m$ is defined on the ball $B_{\bar m}(0)$, for every $m\geq \bar m$. In particular, for every $\bar m$, since $g_m(0)=0$ for each $m\geq 1$, the family $\{g_m|_{B_{\bar{m}}}\}_{m \ge \bar{m}}$ is uniformly bounded and uniformly equicontinuous, thus it is sequentially compact with respect to the uniform convergence. Up to passing to a subsequence identified with a diagonal argument, we can assume that the sequence  $\{g_{m}\}_{m \ge \bar m}$ converges uniformly to a continuous function $\bar g_{\bar{m}}$ on $B_{\bar m}(0)$. Moreover, the fact that the whole sequence has bounded Lipschitz constant (by $2 \Lip(f)$) guarantees that also the limit function $\bar g_{\bar m}$ is a Lipschitz function (with Lipschitz constant at most equal to $2 \Lip(f)$). Therefore, taking the limit as $\bar m \to \infty$ we find a $2\Lip(f)$-Lipschitz function $g \colon (\R^{n(x)}, || \cdot ||) \to \R$ with the property that $g|_{B_{\bar m}(0)} = \bar g_{\bar m}$ for any $\bar m \ge 1$.

Arguing as in \cite[Chapter 10]{MR1708448}, it is possible to prove that, for $\m$-almost every $x\in \mathcal R (\X)$, the function $g$ we obtained is generalized linear, and thus linear accordingly to Lemma \ref{lem:lin}. In particular, $g(z)= \scal{v_g}{z}$, for a suitable vector $v_g \in \R^{n(x)}$. Now, it is easy to realize that, for any suitable $x$, the sequence of maps $\{i_m\}_{m\geq 1}$ defined as
\begin{equation*}
    i_m := \tilde i _m |_{B_{2 r_m}(x)} : (B_{2 r_m}(x), \di_{r_m}) \to (\R^{n(x)}, \norm{\cdot})
\end{equation*}
satisfies the requirements of Proposition \ref{prop:differential}.
\end{proof}

\begin{corollary}\label{cor:kswithdiff}
Under the assumptions of Proposition \ref{prop:differential} it holds that
\begin{equation*}
    \ks^2[f](x) = \fint_{B_1(0)} |g|^2 \de \Leb^{n(x)},
\end{equation*}
for $\m$-almost every $x$.
\end{corollary}

\begin{proof}
First of all, notice that, according to Proposition \ref{prop:convofpotential}, for $\m$-almost every point $x$ where Proposition \ref{prop:differential} holds, we have
\begin{equation*}
    \ks^2[f](x)= \lim_{m\to \infty} \ks^2_{r_m}[f](x) = \lim_{m\to \infty} \fint_{B_{r_m}( x)} \frac{|f(x')-f(x)|^2}{r_m^2} \de \m (x').
\end{equation*}
We can now estimate the right-hand side, shifting the computation on $\R^{n(x)}$ through the maps $\{i_m\}_{m\geq 1}$:
\begin{equation}\label{eq:ksr_m}
    \begin{split}
        \fint_{B_{r_m}(x)} \frac{|f(x')-f(x)|^2}{r_m^2} \de \m(x') &= \int_{B_{r_m}(x)}
        \frac{|f(x')-f(x)|^2}{r_m^2} \de \bar \m(x') \\
        &= \int_{i_m(B_{r_m}(x))} |g_m(z)-g_m(0)|^2 \de [(i_m)_\# \bar \m] (z) \\
        &= \int_{i_m(B_{r_m}(x))} |g_m(z)|^2 \de [(i_m)_\# \bar \m] (z) ,
    \end{split}
\end{equation}
where $\bar \m$ denotes the normalized measure $\frac{\m}{\m(B_{r_m}(x))}$.
On the other hand, since $i_m$ is $\varepsilon_m$-bi-Lipschitz, it holds that
\begin{equation*}
    B_{(1-\varepsilon_m)} (0) \subseteq i_m (B_{r_m}(x)) \subseteq  B_{(1+\varepsilon_m) } (0),
\end{equation*}
then, taking into account \eqref{eq:ksr_m} and property 3 in Proposition \ref{prop:differential}, we can obtain the following estimates:
\begin{equation*}
    \ks^2_{r_m}[f](x) \leq (1+\varepsilon_m)\int_{B_{(1+\varepsilon_m) } (0)} |g_m(z)|^2 \de [c(x)\cdot \Leb^{n(x)}] (z)
\end{equation*}
and 
\begin{equation*}
    \ks^2_{r_m}[f](x) \geq (1-\varepsilon_m)\int_{B_{(1-\varepsilon_m) } (0)} |g_m(z)|^2 \de [c(x)\cdot \Leb^{n(x)}] (z).
\end{equation*}
Combining these two inequalities at the limit $m\to \infty$ we conclude that
\begin{equation*}
\begin{split}
    \ks^2[f](x)= \lim_{m\to \infty} \ks^2_{r_m}[f](x) =  \int_{B_{1 } (0)} |g(z)|^2 \de [c(x)\cdot \Leb^{n(x)}] (z) = \fint_{B_{1 } (0)} |g(z)|^2 \de  \Leb^{n(x)}(z),
\end{split}
\end{equation*}
where the last equality follows from point (ii) in Remark \ref{rmk:deflocMink}.
\end{proof}

\begin{remark}\label{rmk:nonunique}
Our version of Rademacher theorem only proves existence of a differential, which takes the form of a linear function on the tangent space, for $\m$-almost every point. We want to emphasize that, although it is possible to prove relative uniqueness results for the differential in the metric setting (see for example \cite[Theorem 4.38]{MR1708448}), it is impossible to achieve uniqueness in Proposition \ref{prop:differential}. The reason is that our notion of differential is really tailored for spaces satisfying the locally Minkowski assumption, which is strictly local and does not require any consistency on the ``charts" $i^{x,r}$. We clarify this sentence with an example. Consider the metric measure space $(\R^n,\norm{\cdot},\Leb^n)$, where $\norm{\cdot}$ is a $C^1$ norm, which is obviously locally Minkowski. For the origin $0$ we can actually choose the maps $i^{0,r}$ to be proper restrictions of the identity map $\text{Id}$. In this case the differential of the function $f(x)=\scal{v}{x}$ in $0$ will be $f$ itself. However, we can choose the maps $i^{0,r}$ to be proper restrictions of the map $-\text{Id}$ and in this case the differential of in $0$ will be $-f$. Of course we can have also the intermediate situation, where some $i^{0,r}$ are restrictions of $\text{Id}$ and some other are restriction of $-\text{Id}$ (and this happens in particular when $r\to 0$). In this case both $f$ and $-f$ are suitable differentials for $f$ in $0$. As it can be guessed from this example, it could be possible to prove a uniqueness result assuming some consistency property on the maps $i^{x,r}$. However, the existence result provided by Proposition \ref{prop:differential} is sufficient for our purposes.
\end{remark}

\section{Main Result}

In this last section we show the main result of this work, that is the existence on a locally Minkowski metric measure space $(\X, \di, \m)$ of a distance $\di'$  equivalent to $\di$  such that $(\X, \di', \m)$ is infinitesimally Hilbertian. Moreover, for the whole section we assume $(\X, \di, \m)$ to be a length metric measure space, to satisfy the doubling condition and to support a Poincaré inequality. \\

We start by introducing the natural candidate for this distance $\di'$, that is the intrinsic distance associated to the Korevaar-Schoen energy $\Ks$:
\begin{equation}\label{eq:defdks}
    \di_\Ks(x,y)= \sup \{f(y)-f(x)\,:\, f\in \Lipfunc(\X)\, \text{ and }\, \ks[f]\leq 1 \,\, \m-a.e.\}.
\end{equation} 
We are actually going to show that the Cheeger energy $\Ch_\Ks$ associated to the distance $\di_\Ks$ is nothing but the Korevaar-Shoen energy $\tilde \Ks$, which is quadratic (Corollary \ref{cor:ksquadr}). Observe that this distance is not a priori equal to the intrinsic distance associated to the Korevaar-Schoen energy $\tilde \Ks$, but in our case of interest they will turn out to be equal. Therefore, the fact that the definition of $\di_\Ks$ is done considering only Lipschitz function can be seen as a technical choice, which allow us to apply the results developed in the last section. We will go back to this point later, when we will prove the equivalence of the two approaches.

\begin{prop}
$\di_\Ks$ is a distance on the space $\X$.
\end{prop}

\begin{proof}
 The finiteness and the symmetry of $\di_\Ks$ follow directly from its definition together with Corollary \ref{cor:convKS}. Moreover, Corollary \ref{cor:convKS} ensures also that for every $z\in \X$ the function $g(x)= \di(x,z)$ satisfies $\ks[g]\leq 1$ $\m$-almost everywhere, which in particular shows that $\di_\Ks(x,y)=0$ if and only if $x=y$. Let us now prove the triangular inequality: fix $x,y,z\in \X$ and $\varepsilon>0$, then take $f\in \Lipfunc(\X)$ such that $\ks[f]\leq 1$ $\m$-almost everywhere and 
\begin{equation*}
    f(z)- f(x) \geq \di_\Ks(x,z) -\varepsilon.
\end{equation*}
Then we can observe that
\begin{equation*}
    \di_\Ks(x,z) \leq f(z)- f(x) + \varepsilon = [f(z)- f(y)] + [f(y)-f(x)] + \varepsilon \leq \di_\Ks(x,y) + \di_\Ks(y,z)+ \varepsilon,
\end{equation*}
since $\varepsilon$ is arbitrary, we get the conclusion.
\end{proof}

\noindent We are now going to prove that the distances $\di$ and $\di_\Ks$ are equivalent.
Having introduced $\di_\Ks$ working only on Lipschitz functions, Proposition \ref{prop:equivCh} will play a huge role in the proof.

\begin{prop}\label{prop:equivd-dks}
The distance $\di_\Ks$ is equivalent to $\di$.
\end{prop}

\begin{proof}
We start observing that Proposition \ref{prop:equivCh} ensures that $\di_\Ch$ is equivalent to $\di$, so the thesis follows if we prove that $\di_\Ks$ is equivalent to $\di_\Ch$. However, this equivalence is an almost immediate consequence of Proposition \ref{prop:dChLip} and Corollary \ref{cor:convKS}. In fact, let $f \in \Lipfunc(\X)$ be a competitor for the definition of $\di_\Ks$ \eqref{eq:defdks}, that is $\ks[f]\leq 1$ $\m$-almost everywhere. Then the function $\tilde C \cdot f \in \Lipfunc(\X)$ is a competitor for the definition of $\di_\Ch$ \eqref{eq:diChwithLip}, indeed
\begin{equation*}
    \Lip[\tilde C \cdot f] = \tilde C \Lip[f] \leq \ks[f]\leq 1, \quad \m\text{-almost everywhere}.
\end{equation*}
Thus, given any pair $x,y\in \X$, it holds that 
\begin{equation*}
    \di_\Ks(x,y)= \sup_f |f(y)-f(x)| = \frac 1{\tilde C} \sup_f  \big|\tilde C \cdot f (y)- \tilde C \cdot f (x)\big| \leq \frac{1}{\tilde C} \di_\Ch(x,y)
\end{equation*}
where the supremums are taken among the functions $f\in \Lipfunc(\X)$ such that $\ks[f]\leq 1$ $\m$-almost everywhere. Analogously we can prove that $\di_\Ch \leq \di_\Ks$.
\end{proof}

The aim of the next few statement is to prove that $(\X,\di_\Ks)$ is a complete and separable length metric space, notice that completeness and separability are necessary to make $(\X,\di_\Ks,\m)$ a metric measure space (according to Definition \ref{def:mms}). We start with a preliminary Lemma that highlights a nice locality property of the Korevaar-Schoen potentials, this result will be useful many times in the reminder of the section. We point out that also in this proof we take advantage of the convergence of the Korevaar-Shoen potentials (Proposition \ref{prop:convofpotential}).

\begin{lemma}\label{lem:ksofmax}
   Let $f_1, f_2 \in \Lipfunc(\X)$ be such that $\ks[f_1], \ks[f_2]\leq 1 $ $\m$-almost everywhere, then the Lipschitz functions $f=\max \{ f_1,f_2 \}$ and $g= \min \{ f_1,f_2 \}$ satisfy the same property, that is $\ks[f]\leq 1 $ $\m$-almost everywhere and $\ks[g]\leq 1 $ $\m$-almost everywhere.
\end{lemma}

\begin{proof}
We prove the result only for $f$, the proof for $g$ is completely analogous. The function $f$ is clearly Lipschitz being the maximum of two Lipschitz functions. Moreover observe that, since it obviously holds that
\begin{equation*}
    \X= \{ f_1 \geq f_2 \} \cup \{ f_2 \geq f_1 \}
\end{equation*}
it is sufficient to prove that $\ks[f]\leq 1$ for $\m$-almost every $x\in \{ f_1 \geq f_2 \}$ (the same would hold also for $\{ f_2 \geq f_1 \}$ by symmetry). Notice that $\m$-almost every $x\in \{ f_1 \geq f_2 \}$ is a density point of $\{ f_1 \geq f_2 \}$ and it is such that 
\begin{equation*}
    \ks[f](x)=\lim_{r\to 0^+} \ks_r[f](x) \qquad \text{and} \qquad 1 \geq \ks[f_1](x)=\lim_{r\to 0^+} \ks_r[f_1](x),
\end{equation*}
according to Proposition \ref{prop:convofpotential}. In particular we can deduce that 
\begin{align*}
    \ks^2[f](x)&=\lim_{r\to 0^+} \fint_{B_{r}(x)} \frac{|f(y)- f(x)|^2}{r^{2}} \de \m(y)\\
    &=\lim_{r\to 0^+} \frac{1}{\m(B_r(x))} \left[ \int_{B_{r}(x)\cap \{ f_1 \geq f_2 \} } \frac{|f(y)- f(x)|^2}{r^{2}} \de \m(y) + \int_{B_{r}(x)\setminus \{ f_1 \geq f_2 \} } \frac{|f(y)- f(x)|^2}{r^{2}} \de \m(y)  \right] \\
     &=\lim_{r\to 0^+} \frac{1}{\m(B_r(x))} \left[ \int_{B_{r}(x)\cap \{ f_1 \geq f_2 \} } \frac{|f_1(y)- f_1(x)|^2}{r^{2}} \de \m(y) + \int_{B_{r}(x)\setminus \{ f_1 \geq f_2 \} } \frac{|f(y)- f(x)|^2}{r^{2}} \de \m(y)  \right]\\
     &=  \lim_{r\to 0^+} \fint_{B_{r}(x)} \frac{|f_1(y)- f_1(x)|^2}{r^{2}} \de \m(y) = \ks^2[f_1](x) \leq 1.
\end{align*}
This concludes the proof.
\end{proof}

\noindent The next proposition shows that a function of the type $\di_\Ks(\bar z, \cdot)$ is (a posteriori) admissible in the maximization \eqref{eq:defdks}, making the supremum a maximum. Anyway, this is not the application of the result we are interested in, we are going to use it as a building block of some proofs in the following.  

\begin{prop}\label{prop:ksleq1}
 For every $\bar z \in \X$ the function $\rho_{\bar z}=\di_\Ks(\bar z, \cdot)$ is Lipschitz and $\ks[\rho_{\bar z}]\leq 1$ $\m$-almost everywhere.
\end{prop}

\begin{proof}
Consider a countable dense set $\{z_i\}_{i\in \N}\subset \X$, for every $i, m \in \N$ there exists a Lipschitz function $h_{i,m}$, with $\ks[h_{i,m}]\leq 1$ $\m$-almost everywhere, such that 
\begin{equation*}
    h_{i,m}(z_i)- h_{i,m}(\bar z) \geq \di_\Ks(z_i,\bar z) -\frac 1 m .
\end{equation*}
Putting $h_m= \max\{h_{1,m}, \dots, h_{m,m}\}$, Lemma \ref{lem:ksofmax} ensures that $\ks[h_{m}]\leq 1$ $\m$-almost everywhere, moreover, it is easy to realize that the sequence $\{h_m\}_{m\in \N}$ is converging to the function $\rho_{\bar z}$ uniformly on compact sets. In particular $\rho_{\bar z}$ is a Lipschitz function. Now take a 1-Lipschitz smooth function $\phi:[0,\infty)\to [0,1]$ such that $\phi=1$ in $[0,1]$ and $\phi=0$ in $[3,\infty)$.
Notice that the metric version of the Hopf-Rinow theorem (see \cite[Theorem 2.3]{MR1377265}) ensures that $(\X,\di,\m)$ is proper (keep in mind (iv) in Remark \ref{rmk:deflocMink}). Hence for every $R>0$ the functions
\begin{equation*}
    \tilde h_n = h_n \cdot \phi(\di(\bar z,\cdot)/R),
\end{equation*}
converge uniformly to $\tilde \rho_{\bar z}= \rho_{\bar z} \cdot \phi(\di(\bar z,\cdot)/R)$ and they are also equi-Lipschitz. Therefore it is possible to apply Lemma \ref{lem:mazur} and deduce that for every $x\in B_R(\bar z)$ it holds
\begin{equation*}
    \ks[\rho_{\bar z}](x) = \ks[\tilde\rho_{\bar z}](x) \leq \limsup_{i\to \infty} \ks[\tilde h_n](x) = \limsup_{i\to \infty} \ks[ h_n](x)\leq  1.
\end{equation*}
Since this is true for every $R$, $\ks[\rho_{\bar z}]\leq 1$ $\m$-almost everywhere in $\X$.
\end{proof}

\begin{prop}
 $(\X,\di_\Ks)$ is a complete and separable length metric space.
\end{prop}

\begin{proof}
Completeness and separability are immediate consequences of Proposition \ref{prop:equivd-dks}, since $(\X,\di)$ is a complete and separable metric space. Therefore, in order to prove that $\di_\Ks$ is a length metric, it is sufficient to show (see \cite{MR1377265}) the existence of an $\varepsilon$-midpoint for every $x,y\in\X$ and every $\varepsilon>0$, that is $z\in \X$ such that
\begin{equation*}
    \max\{ \di_\Ks(x,z), \di_\Ks(y,z)\} \leq \frac 12 \di_\Ks(x,y) + \varepsilon.
\end{equation*}
Assume by contradiction that this is not the case for some $x,y\in\X$, then there exists $r>\frac{1}{2}\di_\Ks(x,y)$ such that the balls $B_r(x)$ and $B_r(y)$ are disjoint. Define the functions
\begin{equation*}
    \rho_{x,r}(z) = (r- \di_\Ks(x,z))_+ \qquad \text{and} \qquad  \rho_{y,r}(z) = (r- \di_\Ks(y,z))_+ 
\end{equation*}
and notice that, as a consequence of Lemma \ref{lem:ksofmax} and Proposition \ref{prop:ksleq1}, they are Lipschitz and $\ks[\rho_{x,r}]\leq 1$, $\ks[\rho_{x,r}]\leq 1$ $\m$-almost everywhere. As a consequence the function $f= \rho_{x,r}- \rho_{y,r}$ is admissible in \eqref{eq:defdks}, thus
\begin{equation*}
    \di_\Ks(x,y) \geq f(x)- f(y) = \rho_{x,r}(x)- \rho_{y,r}(y) = 2r > \di_\Ks(x,y)
\end{equation*}
which gives the desired contradiction. 
\end{proof}

As already mentioned before, the aim of this section is to prove that the Cheeger energy $\Ch_\Ks$ is equal to the Korevaar-Schoen energy $\tilde \Ks$. The following lemma provides us the strategy to do it, which will be developed in the next two subsections.

\begin{lemma}
   In order to prove $\Ch_\Ks=\tilde \Ks$ it is sufficient to prove that for every Lipschitz function $f \in \Lipfunc(\X)$ it holds
   \begin{equation}\label{eq:lip=ks}
        \ks[f](x) = \Lip_{\di_\Ks}[f] (x) \qquad \text{ for $\m$-almost every }x\in \X.
   \end{equation}
\end{lemma}

\begin{proof}
First of all notice that, since $\di_\Ks \simeq \di$, the Lipschitz functions with respect to the distance $\di_\Ks$ are precisely the Lipschitz functions with respect to the distance $\di$. Moreover, Proposition \ref{prop:equivd-dks} combined with Proposition ensures that $(\X,\di_\Ks,\m)$ satisfies the doubling condition and supports a Poincaré inequality. Then, accordingly to Proposition \ref{prop:Chasrelax}, $\Ch_\Ks$ is the relaxation of 
\begin{equation*}
    \begin{cases}
    \int \Lip_{\di_\Ks}[f]^2 \de \m &\text{if }f\in \Lipfunc(\X)\\
    +\infty & \text{if }f\in L^2(\X,\m) \setminus \Lipfunc(\X)
    \end{cases}.
\end{equation*}
On the other hand, by definition $\tilde \Ks$ is the relaxation of \begin{equation*}
    \begin{cases}
    \int \ks^2[f] \de \m &\text{if }f\in \Lipfunc(\X)\\
    +\infty & \text{if }f\in L^2(\X,\m) \setminus \Lipfunc(\X)
    \end{cases},
\end{equation*}
thus, if \eqref{eq:lip=ks} holds, $\Ch_\Ks$ and $\tilde \Ks$ are the relaxation of the same functional and then they are equal.
\end{proof}

\begin{remark}
Another possible strategy to prove the existence of an equivalent distance which provides an infinitesimally Hilbertian structure on $(\X,\di,\m)$ could be to consider the intrinsic distance $\di_{\tilde \Ks}$ associated to the energy $\tilde \Ks$ (in the sense of Dirichlet forms, see \cite{St95}) and prove that $\Ch_{\di_{\tilde \Ks}}=\tilde \Ks$. A reasonable way to show this last equivalence would be to prove that the energy $\tilde \Ks$ is upper regular (cfr. \cite[Definition 3.13]{MR3298475}) and apply Theorem 3.14 in \cite{MR3298475}. However, proving the upper regularity of the Korevaar-Schoen energy in our setting seemed quite challenging to us, especially in comparison to our strategy. 
\end{remark}

\subsection{Proof of the inequality $\ks[f]\leq \Lip_{\di_\Ks}[f]$}

In this short subsection we prove one of the two inequalities needed to conclude our main result. Let us point out that for this inequality it is not necessary to assume the reference metric measure space $(\X,\di,\m)$ to be locally Minkowski. In particular, our proof follows the strategy already developed by Cheeger in \cite[Section 12]{MR1708448}, with some minor changes which were necessary for it to be applied in our setting. 
Before going to the proof of the inequality, we state a preliminary approximation lemma that was proven by Cheeger in \cite[Theorem 6.5]{MR1708448}.

\begin{lemma}\label{lem:approxch}
    Let $(\X,\mathtt{d},\m)$ be a metric measure space satisfying the doubling condition, for some $\bar z \in \X$ let $f:B_R(\bar z)\to \R$ be a Lipschitz function. Then there exist a sequence of Lipschitz functions $f_i:B_R(\bar z)\to \R$ and, for each $i$, a collection of pointed closed sets $z_{i,l}\in C_{i,l}\subset B_R(\bar z)$, with associated constants $c_{i,l}\geq 0$, such that
    \begin{enumerate}
        \item $f_i \to f$ uniformly and $\Lip_{\mathtt d}[f_i] \to \Lip_{\mathtt d}[f]$ in $L^2(\X,\m)$
        \item $f_i|_{C_{i,l}}= c_{i,l} \cdot \mathtt d(z_{i,l},\cdot)$
        \item $\lim_{i \to \infty} \m(B_R(\bar z)\setminus \cup_l C_{i,l}) = 0 $.
    \end{enumerate}
\end{lemma}

\begin{prop}
 For every Lipschitz function $f \in \Lipfunc(\X)$, it holds that $\ks[f]\leq \Lip_{\di_\Ks}[f]$ $\m$-almost everywhere.
\end{prop}

\begin{proof}
Given a Lipschitz function $f\in \Lipfunc(\X)$ we fix a ball $B_R(\bar z)$ and we apply Lemma \ref{lem:approxch} to the space $(\X,\di_\Ks,\m)$, obtaining an approximating sequence $\{f_i\}_{i\in \N}$ satisfying 1, 2 and 3. Notice that this metric measure space satisfies the doubling condition, as a consequence of the combination of Proposition \ref{prop:equivd-dks} and Proposition \ref{prop:D+P}. Up to taking a subsequence of $\{f_i\}$, we can assume that $\Lip_{\di_\Ks}[f_i] \to \Lip_{\di_\Ks}[f]$ $\m$-almost everywhere and $ \m(B_R(\bar z)\setminus \cup_l C_{i,l})\leq 2^{-i}$ for every $i$. This last requirement implies in particular that $\m$-almost every $x \in B_R(\bar z)$ there exists $\iota(x)$ such that $x\in \cup_l C_{i,l}$ for every $i\geq \iota (x)$. On the other hand notice that for $m$-almost every $x\in C_{i,l}$ (in particular, for every density point of $C_{i,l}$) it holds that
\begin{align*}
    \ks[f_i](x) = \ks[ c_{i,l} \cdot \di_\Ks(z_{i,l},\cdot)] &= c_{i,l} \cdot \ks[\di_\Ks(z_{i,l},\cdot)] \\
    &\leq c_{i,l} = \Lip_{\di_\Ks}[c_{i,l} \cdot \di_\Ks(z_{i,l},\cdot)](x) \leq \Lip_{\di_\Ks}[f_i](x),
\end{align*}
where the first inequality follows from Proposition \ref{prop:ksleq1}, while the first equality and the second inequality hold because $x$ is a density point of $C_{i,l}$. Now we can apply Lemma \ref{lem:mazur} and deduce that for $\m$-almost every $x\in B_R(\bar z)$
\begin{equation*}
     \ks[f](x) \leq \limsup_{i\to \infty} \Lip_{\di_\Ks}[f_i](x) = \Lip_{\di_\Ks}[f](x).
\end{equation*}
Since this can be deduced for every ball $B_R(\bar z)$, we can conclude that $\ks[f]\leq \Lip_{\di_\Ks}[f]$ $\m$-almost everywhere.
\end{proof}

\subsection{Proof of the inequality $\ks[f]\geq \Lip_{\di_\Ks}[f]$}
In this last subsection we prove the most challenging inequality needed to conclude \eqref{eq:lip=ks} and consequently our main result $\Ch_\Ks=\tilde \Ks$. We stress that the locally Minkowski assumption on the reference metric measure space will play a crucial role. In fact, it will allow us to translate (in some sense) our problem to an Euclidean space, where a standard duality result (Lemma \ref{lem:duality}) will be enough to conclude. For this reason, in the following we will often deal with the Korevaar-Shoen energy of linear functions in $\R^{n(x)}$, thus we adopt the following notation:
\begin{equation*}
     \ks^2[v]:=\fint_{B_1(0)}  |\scal{v}{\cdot}|^2 \de \Leb^{n(x)} .
\end{equation*}
In particular, for every fixed Lipschitz function $f\in \Lipfunc(\X)$, we are going to prove that 
\begin{equation}\label{eq:estimate2}
    \ks[f](x)\geq \Lip_{\di_\Ks}[f](x), 
\end{equation}
for every $x\in \X$ such that the conclusion of Proposition \ref{prop:differential} holds in $x$ for a linear function $g=\scal{v_g}{\cdot}$ and that
\begin{equation}\label{eq:lip+}
    \Lip_{\di_\Ks}[f](x) = \lim_{r \to 0} \sup_{\di_\Ks(y,x)=r} \frac{|f(y)-f(x)|}{\di_\Ks(y,x)} .
\end{equation}
Notice that, as a consequence of Proposition \ref{prop:differential} and Proposition \ref{prop:pointlip}, this set of $x$ has full $\m$-measure.  
Fixed such an $x$, accordingly to Corollary \ref{cor:kswithdiff}, we have that 
\begin{equation*}
    \ks^2[f](x) = \fint_{B_1(0)} |g|^2 \de \Leb^{n(x)} = \ks^2[v_g].
\end{equation*}
Take $0<d<\frac 14 \frac{\sqrt{\tilde C}}{1 + \sqrt{\tilde C}}$, where $\tilde C$ is the constant appearing in Corollary \ref{cor:convKS} (and in Proposition \ref{prop:potestimate}). Recall that $\di_\Ks \simeq \di$ (cfr. Proposition \ref{prop:equivd-dks}), then we can find $0<a<b<c<d$ such that for every $s>0$
\begin{equation*}
    \{x' \, :\,\di_\Ks(x',x)=bs\} \subset \{x' \, :\,a s\leq \di(x',x)\leq cs\}:= C_s \subset \{x' \, :\,\di_\Ks(x',x)\leq d s\}.
\end{equation*}
Consequently, keeping in mind \eqref{eq:lip+}, we deduce that 
\begin{align*}
    \Lip_{\di_\Ks}[f](x) = \lim_{m \to 0} \sup_{\di_\Ks(y,x)=br_m} \frac{|f(y)-f(x)|}{\di_\Ks(y,x)} & \leq \lim_{m \to 0} \sup_{y \in C_{r_m}} \frac{|f(y)-f(x)|}{\di_\Ks(y,x)} \\
    &\leq \lim_{m \to 0} \sup_{\di_\Ks(y,x)\leq d r_m} \frac{|f(y)-f(x)|}{\di_\Ks(y,x)} = \Lip_{\di_\Ks}[f](x),
\end{align*}
and in particular 
\begin{equation*}
    \Lip_{\di_\Ks}[f](x) = \lim_{m \to 0} \sup_{y \in C_{r_m}} \frac{|f(y)-f(x)|}{\di_\Ks(y,x)}.
\end{equation*}
Now, for every fixed $\varepsilon>0$, we can find $m$ such that $\varepsilon_m<\varepsilon$,
\begin{equation}\label{eq:linftyest}
    \norm{\frac{(f-f(x))\circ (i_m)^{-1}}{r_m}-g}_{L^\infty(B_1(0))}< \varepsilon,
\end{equation}  
and 
\begin{equation}\label{eq:lipCrm}
    \Lip_{\di_\Ks}[f](x) \leq  \sup_{ y \in C_{r_m}} \frac{|f(y)-f(x)|}{\di_\Ks(y,x)} +\varepsilon.
\end{equation}
In particular, from \eqref{eq:linftyest} we deduce that 
\begin{equation}\label{eq:estimateg}
    \frac{|f(y)-f(x)|}{r_m}\leq |g(i_m(y))| +\varepsilon,
\end{equation}
for every $y \in (i_m)^{-1} (B_1(0)) \supset (i_m)^{-1} (B_{2d}(0)) \supset B_{d r_m}(x) \supset C_{r_m}$. Consequently, combining \eqref{eq:lipCrm} and \eqref{eq:estimateg},  we have that
\begin{equation}\label{eq:prelemma}
    \begin{split}
        \Lip_{\di_\Ks}[f](x) &\leq  \sup_{  y \in C_{r_m}} \frac{r_m\cdot |g(i_m(y))|+ r_m \varepsilon}{\di_\Ks(y,x)}  +\varepsilon\\
        &\leq  \sup_{  y \in C_{r_m}} \frac{r_m\cdot |g(i_m(y))|}{\di_\Ks(y,x)} + \frac \varepsilon e +\varepsilon\\
    &\leq \sup_{z\in B_{2d}(0) } \frac{r_m\cdot |g(z)|}{\di_\Ks(x,(i_m)^{-1}(z))} + \frac\varepsilon e +\varepsilon\\
    &\leq \sup_{z\in B_{2d}(0) } \frac{r_m\cdot |\scal{v_g}{z}|}{\di_\Ks(x,(i_m)^{-1}(z))} + \frac\varepsilon e +\varepsilon,
    \end{split}
\end{equation}
where the constant $e>0$ is such that $\di_\Ks(x,\cdot) \geq e r_m$ on $C_{r_m}$.
The following lemma helps to estimate the supremum in the last term of \eqref{eq:prelemma}.

\begin{lemma}
  Let $z\in B_{2d}(0)$, then 
\begin{equation}\label{eq:finallemma}
    \di_\Ks (x,(i_m)^{-1}(z)) \geq (1-O(\varepsilon)) \cdot r_m \cdot \sup_{w\in \R^{n(x)}}\frac{|\scal{w}{z}|}{\ks[w]}
\end{equation}
\end{lemma}

\begin{proof}
Fix a linear function $h=\scal{w}{\cdot}$ on $\R^{n(x)}$ and define the function
\begin{equation*}
    f: B_{2r_m}(x) \to \R \qquad f_m := r_m \cdot h \circ i_m ,
\end{equation*}
which is clearly Lipschitz and satisfies $f_m(x)=0$. 
Fix a point $y\in B_{\frac32 r_m}(x)$ 
and notice that $B_r(y)=B_{r/r_m}^{\di_{r_m}}(y)$, then for $r$ small enough
\begin{align*}
    \ks^2_r[f_m](y) &= \fint_{B_r(y)} \frac{|f_m(y')-f_m(y)|^2}{r^2} \de \m (y')= \frac{1}{r_m^2} \frac{1}{\m(B_r(x))}\int_{{B_{r/r_m}^{\di_{r_m}}(y)}} \frac{|f_m(y')-f_m(y)|^2}{(r/r_m)^2} \de \m (y')\\
    &=  \frac{1}{\m(B_r(x))} \int_{i_m(B_{r/r_m}^{\di_{r_m}}(y))} \frac{|h(z')-h(z)|^2}{(r/r_m)^2} \de (i_m)_\#\m (z'),
\end{align*}
where $z=i_m(y)$. Moreover, if we denote $\bar \m$ the normalized measure $\frac{\m}{\m(B_{r_m}(x))}$, we also have that
\begin{equation}\label{eq:ksmanipulated}
    \ks^2_r[f_m](y) = \frac{1}{\bar\m(B_r(x))} \int_{i_m(B_{r/r_m}^{\di_{r_m}}(y))} \frac{|h(z')-h(z)|^2}{(r/r_m)^2} \de (i_m)_\#\bar \m (z').
\end{equation}
Now, since $i_m$ is $\varepsilon_m$-bi-Lipschitz, it holds that 
\begin{equation}\label{eq.setinclusion}
    B_{(1-\varepsilon_m) r/r_m} (z) \subseteq i_m (B_{r/r_m}^{\di_{r_m}}(y)) \subseteq  B_{(1+\varepsilon_m) r/r_m} (z),
\end{equation}
and consequently that
\begin{align*}
    (1-\varepsilon_m)c(x)& \cdot \Leb^{n(x)} (B_{(1-\varepsilon_m) r/r_m} (z)) \leq (i_m)_\# \bar\m (B_{(1-\varepsilon_m) r/r_m} (z)) \leq \bar\m (B_{r/r_m}^{\di_{r_m}}(y))\\
    &= \bar\m(B_r(x)) 
    \leq (i_m)_\# \bar\m (B_{(1+\varepsilon_m) r/r_m} (z))\leq (1+\varepsilon_m)c(x) \cdot \Leb^{n(x)} (B_{(1+\varepsilon_m) r/r_m} (z)).
\end{align*}
Using the last inequality and the scaling property of the Lebesgue measure, we deduce that 
\begin{equation*}
    c(x) \cdot \Leb^{n(x)} (B_{(1+\varepsilon_m) r/r_m} (z)) = \bigg[\frac{1+\varepsilon_m}{1-\varepsilon_m}\bigg]^{n(x)} c(x) \cdot \Leb^{n(x)} (B_{(1-\varepsilon_m) r/r_m} (z)) \leq \frac{(1+\varepsilon_m)^{n(x)}}{(1-\varepsilon_m)^{n(x)+1}} \bar \m(B_r(x)).
\end{equation*}
Combining this last inequality with \eqref{eq:ksmanipulated} we obtain 
\begin{align*}
    \ks^2_r[f_m](y) &\leq \frac{(1+\varepsilon_m)^{n(x)}}{(1-\varepsilon_m)^{n(x)+1}} \frac{1}{  c(x) \cdot \Leb^{n(x)} (B_{(1+\varepsilon_m) r/r_m} (z))} \int_{i_m(B_{r/r_m}^{\di_{r_m}}(y))} \frac{|h(z')-h(z)|^2}{(r/r_m)^2} \de (i_m)_\#\bar \m (z') \\
    &\leq \bigg[\frac{1+\varepsilon_m}{1-\varepsilon_m}\bigg]^{n(x)+1}\frac{1}{\Leb^{n(x)} (B_{(1+\varepsilon_m) r/r_m} (z))} \int_{B_{(1+\varepsilon_m) r/r_m} (z)} \frac{|h(z')-h(z)|^2}{(r/r_m)^2} \de \Leb^{n(x)} (z') \\
    & = \frac{(1+\varepsilon_m)^{n(x)+3}}{(1-\varepsilon_m)^{n(x)+1}}\fint_{B_{(1+\varepsilon_m) r/r_m} (z)} \frac{|h(z')-h(z)|^2}{(1+\varepsilon_m)^2(r/r_m)^2} \de \Leb^{n(x)} (z')\\
    & = \frac{(1+\varepsilon_m)^{n(x)+3}}{(1-\varepsilon_m)^{n(x)+1}}\fint_{B_1(0)} |h(z')|^2 \de \Leb^{n(x)} (z') = (1+O(\varepsilon_m))\cdot \ks^2[w],
\end{align*}
where the second inequality follows from \eqref{eq.setinclusion} and the last equality holds because $h$ is a linear function. Since the last inequality holds for every $y\in B_{\frac32 r_m}(x)$ and every $r$ sufficiently small we conclude that
\begin{equation*}
    \ks[f_m] \leq (1+O(\varepsilon_m))\cdot \ks[w],
\end{equation*}
for $\m$-almost every $y\in B_{\frac32 r_m}(x)$. In particular there exists $\mathsf K$ (independent from $m$) such that, the Lipschitz function 
\begin{equation*}
    \tilde f_m= \frac{f_m}{(1+\mathsf K\varepsilon_m)\cdot \ks[w]} 
\end{equation*}
satisfies $\ks[\tilde f_m]\leq 1$, $\m$-almost everywhere on $B_{\frac32 r_m}(x)$. We then define the function $\psi:\X \to \R$ as  $\psi(x')= \max\{r_m- \di(x,x'),0\}$ and consider the function
\begin{equation*}
    \bar{f}_m = \max\{\min\{\tilde f_m, \psi\},-\psi\}.
\end{equation*}
Observe that the function $\bar f_m$ is Lipschitz and it is constantly equal to 0 outside the ball $B_{r_m}(x)$. Moreover, notice that, according to Corollary \ref{cor:convKS}, we have $\ks[\psi]\leq\Lip[\psi]\leq 1$, $\m$-almost everywhere in $\X$, then, applying Lemma \ref{lem:ksofmax}, we deduce that $\ks[\bar f_m]\leq 1$ $\m$-almost everywhere. In particular $\bar f_m$ is a competitor for \eqref{eq:defdks}. On the other hand, applying Corollary \ref{cor:convKS} once again, we deduce that $\Lip[\tilde f_m]\leq \frac{1}{\sqrt{\tilde C}} \ks[\tilde f_m]\leq \frac{1}{\sqrt{\tilde C}}$ $\m$-almost everywhere and thus $\tilde f_m$ is a $\frac{1}{\sqrt{\tilde C}}$-Lipschitz function. Now, since $\psi$ is $1$-Lipschitz, $\psi(0)=r_m$ and $\tilde f_m(x)=0$, we have that $\bar f_m= \tilde f_m$ on the set 
\begin{equation*}
   i_m^{-1}(B_{2d}(0)) \subset B^{\di_{r_m}}_{4d}(x) \subset B_{ \frac{\sqrt{\tilde C}}{1 + \sqrt{\tilde C}} r_m}(x).
\end{equation*}
In particular, by the definition of the intrinsic distance, we conclude that for every $z\in B_{2d}(0)$ 
\begin{align*}
    \di_\Ks (x,(i_m)^{-1}(z)) &\geq |\bar f_m((i_m)^{-1}(z))- \bar f_m(x)| = |\tilde f_m((i_m)^{-1}(z))- \tilde f_m(x)| \\
    &= \frac{1}{(1+\mathsf K\varepsilon_m)\cdot \ks[w]} | f_m((i_m)^{-1}(z))-  f_m(x)| = \frac{r_m |\scal{w}{z}| }{(1+\mathsf K\varepsilon_m)\cdot \ks[w]}.
\end{align*}
Taking the supremum over all linear functions $h$, and thus over all $w\in \R^{n(x)}$, we obtain \eqref{eq:finallemma}.
\end{proof}

\noindent Now we can put together the result of this last lemma with \eqref{eq:prelemma}, obtaining
\begin{equation*}
    \Lip_{\di_\Ks}[f](x) \leq (1+O(\varepsilon))\sup_{z\in B_{2d}(0) } \frac{|\scal{v_g}{z}|}{\sup_{w\in \R^{n(x)}}\frac{|\scal{w}{z}|}{\ks[w]}} + \frac\varepsilon e +\varepsilon.
\end{equation*}
Then we can send $\varepsilon \to 0$ and conclude 
\begin{equation}\label{eq:finalest}
    \Lip_{\di_ \Ks}[f](x) \leq \sup_{z\in B_{2d}(0) } \frac{|\scal{v_g}{z}|}{\sup_{w\in \R^{n(x)}}\frac{|\scal{w}{z}|}{\ks[w]}}= \sup_{z\in  \R^{n(x)}} \frac{|\scal{v_g}{z}|}{\sup_{w\in \R^{n(x)}}\frac{|\scal{w}{z}|}{\ks[w]}}.
\end{equation}
At this point we have shifted our problem on $\R^{n(x)}$ and in order to conclude it is sufficient the following lemma.

\begin{lemma}\label{lem:duality}
   For every $v\in \R^{n(x)}$ it holds that
   \begin{equation*}
       \sup_{z\in \R^{n(x)} } \frac{|\scal{v}{z}|}{\sup_{w\in \R^{n(x)}}\frac{|\scal{w}{z}|}{\ks[w]}} = \ks[v]
   \end{equation*}
\end{lemma}

\begin{proof}
We can explicit $\ks[v]$ in coordinate, using the Einstein notation and obtaining that 
\begin{align*}
     \ks^2[v]&=\fint_{B_1(0)}  |\scal{v}{z}|^2 \de \Leb^{n(x)}(z) =  \fint_{B_1(0)} ( v_i z^i )^2 \de \Leb^{n(x)}(z) \\
     &= \fint_{B_1(0)}  v_i v_j z^i z^j  \de \Leb^{n(x)}(z)= v_i A^{i,j} v_j = v^t A v,
\end{align*}
where $A^{i,j}=\fint_{B_1(0)}  z^i z^j  \de \Leb^{n(x)}(z)$. The matrix $A$ is obviously symmetric and it is easy to notice that is positive definite, in fact 
\begin{equation*}
    v^t A v = \ks^2[v] >0 \qquad \text{for every }v\neq 0,
\end{equation*}
in particular it is invertible with symmetric inverse $A^{-1}$. As a consequence we can deduce that 
\begin{align*}
    \sup_{w\in \R^{n(x)}}\frac{|\scal{w}{z}|}{\ks[w]} = \sup_{w\in \R^{n(x)}} \frac{w^t A (A^{-1} z)}{(w^t A w)^{1/2}}=[(A^{-1} z)^t A (A^{-1} z)]^{1/2}= (z^t A^{-1} z)^{1/2}.
\end{align*}
With the same argument we obtain the following identities
\begin{equation*}
    \sup_{z\in \R^{n(x)} } \frac{|\scal{v}{z}|}{\sup_{w\in \R^{n(x)}}\frac{|\scal{w}{z}|}{\ks[w]}} =  \sup_{z\in \R^{n(x)} } \frac{|\scal{v}{z}|}{(z^t A^{-1} z)^{1/2}} = (v^t A v)^{1/2} = \ks[v],
\end{equation*}
which allow to conclude.
\end{proof}

\noindent In particular, applying Lemma \ref{lem:duality} and keeping in mind \eqref{eq:finalest}, we conclude that 
\begin{equation*}
     \Lip_{\di_\Ks}[f](x) \leq \ks[v_g] = \ks[f](x),
\end{equation*}
which is the desired inequality. 

In conclusion, we have proven that the Korevaar-Schoen energy $\tilde \Ks$ is the Cheeger energy associated to the metric measure space $(\X,\di_\Ks, \m)$. Since the energy $\tilde \Ks$ is quadratic (cfr. Corollary \ref{cor:ksquadr}), $(\X,\di_\Ks, \m)$ is an infinitesimally Hilbertian metric measure space. Moreover the same argument  as in section \ref{sec:diCh} allows to prove that $\di_\Ks$ coincide with the the intrinsic distance associated to the energy $\tilde \Ks$.



\end{document}